\documentclass[11pt,a4paper]{amsart}
\usepackage{amsmath,amssymb,amscd,amsthm,amsxtra, mathrsfs, array,soul,ulem}

\usepackage[latin9]{inputenc}

\usepackage[yyyymmdd,hhmmss]{datetime}
\usepackage[usenames,dvipsnames]{xcolor}
\usepackage[
colorlinks=true,
linkcolor=blue,
citecolor=blue,
urlcolor=blue,
pagebackref,
]{hyperref}

\usepackage{ulem}
\usepackage{marginnote}

\newtheorem{theorem}{Theorem}[section]
\newtheorem{lemma}[theorem]{Lemma}
\newtheorem{proposition}[theorem]{Proposition}
\newtheorem{corollary}[theorem]{Corollary}
\newtheorem{claim}[theorem]{Claim}

\theoremstyle{definition}
\newtheorem{definition}[theorem]{Definition}

\newtheorem{example}[theorem]{Example}

\theoremstyle{remark}
\newtheorem{remark}[theorem]{Remark}
\numberwithin{equation}{section}

\newcommand{\absval}[1]{\mbox{$|#1|$}}

\def\R{{\mathbb R}}

\newcommand{\cQ }{\mathcal Q}
\newcommand{\cF }{\mathcal F}
\newcommand{\cR }{\mathcal R}
\newcommand{\ccR }{\mathfrak{R}}

\newcommand{\x }{{\bf x}}

\newcommand{\norm}[1]{\mbox{$\left\| #1 \right\|$}}

\newcommand{\BMO}[0]{\operatorname{BMO}}

\def\Xint#1{\mathchoice
  {\XXint\displaystyle\textstyle{#1}}%
  {\XXint\textstyle\scriptstyle{#1}}%
  {\XXint\scriptstyle\scriptscriptstyle{#1}}%
  {\XXint\scriptscriptstyle\scriptscriptstyle{#1}}%
  \!\int}
\def\XXint#1#2#3{{\setbox0=\hbox{$#1{#2#3}{\int}$}
    \vcenter{\hbox{$#2#3$}}\kern-.5\wd0}}

\def\avgint{\Xint-}
\newcommand{\vertiii}[1]{{\left\vert\kern-0.25ex\left\vert\kern-0.25ex\left\vert #1 
    \right\vert\kern-0.25ex\right\vert\kern-0.25ex\right\vert}}

\numberwithin{equation}{section}

\definecolor{caro}{rgb}{0.0, 0.5, 0.0}

\begin{document}

 \title[ Self-improving on rectangles]{Self-improving Poincar\'e-Sobolev type functionals in product spaces}

\author[M.E. Cejas]{Mar\'ia Eugenia Cejas}
\address[Mar\'ia Eugenia Cejas]{Departamento de Matem\'atica, Facultad de Ciencias Exactas, Universidad Nacional de La Plata, Calle 50 y 115, (1900) La Plata, Prov. de Buenos Aires, Argentina.}
 \email{ecejas@mate.unlp.edu.ar}

\author[C. Mosquera]{Carolina Mosquera}
\address[Carolina Mosquera]{Department of Mathematics,
Facultad de Ciencias Exactas y Naturales,
University of Buenos Aires, Ciudad Universitaria
Pabell\'on I, Buenos Aires 1428 Capital Federal Argentina} \email{mosquera@dm.uba.ar}

\author[C. P\'erez]{Carlos P\'erez}
\address[Carlos P\'erez]{ Department of Mathematics, University of the Basque Country, IKERBASQUE 
(Basque Foundation for Science) and
BCAM \textendash  Basque Center for Applied Mathematics, Bilbao, Spain}
\email{cperez@bcamath.org}

\author[E. Rela]{Ezequiel Rela}
\address[Ezequiel Rela]{Department of Mathematics,
Facultad de Ciencias Exactas y Naturales,
University of Buenos Aires, Ciudad Universitaria
Pabell\'on I, Buenos Aires 1428 Capital Federal Argentina} \email{erela@dm.uba.ar}

\thanks{M.E.C. is partially supported by grant PICT-2018-03017 (ANPCYT)}
\thanks{C.M. is partially supported by grants UBACyT 20020170100430BA, PICT 2018--03399 and PICT 2018--04027.}
\thanks{C. P. is supported by grant  MTM2017-82160-C2-1-P of the Ministerio de Econom\'ia y Competitividad (Spain), grant IT-641-13 of the Basque Government, and IKERBASQUE}
\thanks{E.R. is partially supported by grants UBACyT  20020170200057BA, PIP (CONICET) 11220110101018, by the Basque Government through the BERC 2014-2017 program, and by the Spanish Ministry of Economy and Competitiveness MINECO: BCAM Severo Ochoa accreditation SEV-2013-0323.
This project has received funding from the European
Union's Horizon 2020 research and innovation programme under the Marie Sklodowska-Curie grant agreement No 777822. }


\begin{abstract}
In this paper we give a geometric condition which ensures that  $(q,p)$-Poincaré-Sobolev  inequalities are implied from  generalized $(1,1)$-Poincaré inequalities related to $L^1$ norms in the context of product spaces. The concept of eccentricity plays a central role in the paper. We provide several $(1,1)$-Poincar\'e type inequalities adapted to different geometries and then show that our selfimproving method can be applied to obtain special interesting Poincar\'e-Sobolev estimates. Among other results, we prove that  for each rectangle $R$
of the form $R=I_1\times I_2 \subset \mathbb{R}^{n}$ where $I_1\subset \mathbb{R}^{n_1}$ and $I_2\subset \mathbb{R}^{n_2}$ are cubes with sides parallel to the coordinate axes, we have that
\begin{equation*}
\left(  \frac{1}{w(R)}\int_{ R }   |f -f_{R}|^{p_{\delta,w}^*}     \,wdx\right)^{\frac{1}{p_{\delta,w}^*}} \leq 
c\,(1-\delta)^{\frac1p}\,[w]_{A_{1,\ccR}}^{\frac1p}\, \Big(a_1(R)+a_2(R)\Big),
\end{equation*}
where $\delta \in (0,1)$,  $w \in A_{1,\ccR}$,  $\frac{1}{p} -\frac{1}{ p_{\delta,w}^* }=   \frac{\delta}{n} \, \frac{1}{1+\log [w]_{A_{1,\ccR}}}$
and  $a_i(R)$ are bilinear analog of the fractional  Sobolev seminorms $[u]_{W^{\delta,p}(Q)}$ (See Theorem \ref{thm:BBMbiparametrico}). 
This is a biparameter weighted version of the celebrated fractional Poincaré-Sobolev estimates with the gain $(1-\delta)^{\frac1p}$ due to Bourgain-Brezis-Minorescu.
\end{abstract}
\maketitle

\normalem

%\tableofcontents

\section{Introduction and Background }

Poincaré and Poincaré-Sobolev inequalities have been studied extensively in a wide variety of scenarios, including the general theory for convex sets. The particular case of cubes is specially interesting, once we have a good result on cubes there are mechanisms  (sometimes chaining methods) to pass to more general domains, see for instance Chapter 9 from the very recent manuscript \cite{KLV} (also \cite{HKST}). See also 
the last part  of Section 5 in \cite{FPW98} to see the connection with weights.
The classical technique for getting this kind of inequalities is through the use of a
representation formula in terms of a fractional integral. An approach which avoids the use of any representation formula to obtain Poincaré-Sobolev inequalities on cubes (or balls) was introduced in \cite{FPW98} and then sharpened in 
\cite{MacManus-Perez-98}. This approach merges with the John-Nirenberg theory of the $\BMO$ spaces and others under the phenomenon of the so called self-imroving property.  
However this ``abstract" general theory has not been developed in the context of rectangles or, more generally, product spaces as was considered for first time in \cite{ShiTor} and later in \cite{LuWheeden}, which will be the main purpose of this article. In this different scenario it is reasonable to expect some sensitivity with respect to very thin rectangles which is reflected in the dependence on the notion of {\it eccentricity}: roughly, a quantity that indicates how far the given object is from being a cube. This notion will be essential to our results and we will properly define it later.

We will study generalized  weighted Poincar\'e and Poincaré-Sobolev type inequalities related to two product space settings.  The most general possible situation  is the case of the family  of rectangles in $\mathbb{R}^n$ seen as $n$-fold product of intervals on $\mathbb{R}$, denoted by $\cR$.
We will study a variant of this case by considering rectangles $R$ of the form $R=I_1\times I_2 \subset \mathbb{R}^{n}$ where $I_1\subset \mathbb{R}^{n_1}$ and $I_2\subset \mathbb{R}^{n_2}$ are cubes with sides parallel to the coordinate axes, denoted by $\ccR$. We will deal with these two situations separately  to simplify the presentation. Our approach consists in starting with some generalized unweighted Poincaré inequality with respect to an $L^1$ mean oscillation and from there we will improve it in order to gain higher integrability with the minimum loss in the involved constants.

The motivation to study this situation comes from the following very interesting $(1,1)$-Poincar\'e inequality valid for any 
bounded convex domain $\Omega\subset \mathbb{R}^n$, 
\begin{equation}\label{eq:Poincare-(1,1)}
\avgint_\Omega |f-f_\Omega|dx \leq \frac12\, d(\Omega) \avgint_\Omega |\nabla f|dx,
\end{equation}
where  we used $f_\Omega$ and $\avgint_\Omega f \ dx$ to denote the average of the function $f$ over the set $\Omega$ and $d(\Omega)$ stands for the diameter of the set $\Omega$. This result is due to G. Acosta and R. Dur\'an in \cite{AD04} where the constant $\frac12$ in front is best possible. Of course, it is a classical well-known fact that this estimate holds for any convex set $\Omega$ in $\mathbb{R}^n$ with a larger constant in front. 
In particular,  \eqref{eq:Poincare-(1,1)}  holds in the case of $\Omega$ being a rectangle of $\mathbb{R}^n$ with sides parallel to the axes.
The key idea here is that inequality \eqref{eq:Poincare-(1,1)} enjoys a sort of self-improving property, meaning that we can reach higher integrability and therefore obtain $(p,p)$-Poincar\'e and $(p^*,p)$-Poincar\'e-Sobolev type inequalities, with $p^*>p$, both in the unweighted and the weighted settings. This self-improving phenomena is not new and can be traced back to \cite{SC,HaK} in a more classical setting, but in a more general abstract way in \cite{FPW98}  improved in \cite{MacManus-Perez-98}. A strengthened version of the results from these papers 
can be found in \cite{PR-Poincare}  where much more precise $L^q-L^p$ weighted Poincar\'e-Sobolev inequalities over cubes in $\mathbb{R}^n$ were obtained. To be more precise 
it is remarkable that  the self-improving property is not so much related to the presence of a gradient on the right hand side of \eqref{eq:Poincare-(1,1)}, but to a discrete summation condition associated to the functional that it defines.  %introducir los rectangulos
In fact, we will consider as a starting point estimates of the form

\begin{equation}\label{eq:general-starting-a(R)}
\avgint_R |f-f_R|dx\le a(R)\qquad R\in \cF
\end{equation}
where $\cF$ is a given family of sets and $a:\cF \to (0,\infty)$ is some general functional with no restriction. Typical examples of $\cF$ are the family of cubes $\cQ$, the families of rectangles from $\cR$ or $\ccR$, or any  other collection of sets. The key difference when moving from the basis $\cR$ to the basis $\ccR$ is not on the selfimproving phenomena, as it will be clear from the proofs. The issue will be to find the appropriate analogue for \eqref{eq:general-starting-a(R)}, where the functional $a(R)$ will be replaced by a \emph{sum} of two functionals $a_1(I)+a_2(J)$, with $R=I\times J$. Part of the effort will be devoted to prove these starting points where the oscillation is controlled by a sum of functionals.

This kind of generalized Poincaré inequality with a general abstract functional acting on balls (or cubes) 
over a metric (or quasi-metric) space with a doubling measure (spaces of homogeneous type) was considered for first time in \cite{FPW98}.  
A special geometric condition, denoted by $D_p$, was introduced with the idea of deriving self-improving properties  of the given mean oscillation of the function on the left hand side.  This condition has been recently refined in \cite{PR-Poincare}, denoted $SD_p$, which produces much more accurate results.
We introduce both conditions $D_p$ and $SD_p$ for the families $\cR$ and $\ccR$  in Definition \ref{def:Dp(w)} and Definition \ref{def:smallness} below
 in the context of rectangles.  We will make special emphasis on the $SD_p$ type condition since it is highly efficient allowing to improve the estimates obtained in \cite{FPW98} in a much more precise way.  
   
To present here a preview of the main contributions of this paper, consider the particular case of $a(R)$ being the average of the $L^1$ norm of the gradient of $f$, namely $a(R)=d(R) \avgint_R |\nabla f|$. The main goal here would be, starting from an inequality of the form \eqref{eq:general-starting-a(R)} and using certain discrete summability conditions satisfied by the functional $a$, to obtain an inequality involving $L^q-L^p$ norms of the form
\begin{equation*}
\left (\frac{1}{w(R)}\int_R|f-f_R|^{q}w\right )^\frac{1}{q}\le C_w\, d(R)\left (\frac{1}{w(R)}\int_R |\nabla f|^p w\right )^\frac{1}{p}
\end{equation*}
where $w$ is a certain weight. We will try to reach the best possible $q$ but keeping controlled the constant $C_w$ 
extending and improving the results in \cite{FPW98} and  \cite{PR-Poincare}.

The precise definition that we need is the following. As usual, we will say that a \emph{weight} is a non negative locally integrable function defined on $\mathbb{R}^n$. Also, the Lebesgue measure for a measurable set $\Omega$ in $\mathbb{R}^n$ is denoted by $|\Omega|$.

\begin{definition} \label{def:smallness}
%Let $L>1$ and $R\in \cR$. 
For a given weight $w$, and $s>0$, we say that a functional $a\in SD_{p,\cR}^{s}(w)$ for $0< p<+\infty$ if there is a constant $c$ such that for any family 
of disjoint \emph{dyadic} subrectangles $\left\lbrace R_i\right\rbrace_i $ of $R$ the following inequality holds
\begin{equation}\label{eq:SDp}
\left( \sum_i a(R_i)^p \frac{w(R_i)}{w(R)}  \right)^{\frac1p} \leq c 
\left(\frac{|\bigcup_iR_i|}{|R|} \right )^{ \frac1s }
%\left( \frac{1}{L}\right)^{\frac{1}{s}}
a(R).
\end{equation}
\end{definition}
The best possible constant $c$ (the infimum of the constants in last inequality) is denoted by 
$\|a\|_{SD_{p,\cR}^s}$. Usually, if there is no chance of confusion, we will drop the subscript for the sake of clarity in the exposition. 

Sometimes we will refer to this condition as the smallness preservation condition. The idea comes from the main application, namely if $L>1$ and if $R$ is a rectangle,  and a family of dyadic rectangles is $L$-small, namely that 
\begin{equation}\label{eq:smallness}
\frac{|\bigcup_iR_i|}{|R|} \leq  \frac1L.
\end{equation}
then the functional $a$ preserves the smallness by applying \eqref{eq:SDp}.

\begin{remark}
As it can be easily verified, the exact same definition applies to the family $\ccR$ using the appropriate notation $ SD_{p,\ccR}^s(w)$ and $\|a\|_{SD_{p,\ccR}^s}.$
\end{remark}

We also recall the following geometric type condition which is somehow a ``rough" version of \eqref{eq:SDp}. Again, the same definition applies to $\ccR$ with the obvious notational modifications.

\begin{definition} \label{def:Dp(w)}
Let $w$ be any weight. We say that the functional $a$ satisfies the weighted $D_{p,\cR}(w)$ condition for $0<p<\infty$ if there is a constant $c$ such that for any $R \in \cR$ and any family of disjoint \emph{dyadic} subrectangles $\left\lbrace R_i\right\rbrace_i$ of $R$ the following inequality holds:
\begin{equation}\label{eq:Dp}
\left( \sum_i a(R_i)^p \frac{w(R_i)}{w(R)}  \right)^{\frac1p} \leq c\,a(R).
\end{equation}
The best possible constant $c$ above is denoted by $\|a\|_{D_{p,\cR}(w)}$. We will write in this case that $a\in D_{p,\cR}(w)$. Observe that $\|a\|_{D_{p,\cR}(w)}\geq 1$.
\end{definition}

Condition \eqref{eq:Dp} in the context of cubes or balls in $\mathbb{R}^n$,  or more generally in the context of Spaces of Homogeneous type  was introduced in \cite{FPW98}. It was also considered in \cite{OP-nondoubling}  in the ``non-homogeneous" context where an abstract self-improving property was obtained.

We conclude this introduction by remarking how condition \eqref{eq:Dp} can be used, in the much easier context of cubes $\cQ$ with no weight,  to derive very sharp proofs in the theory of fractional Sobolev spaces. Indeed,  in this paper, Corollary \ref{cor:fractional}
and many others, are extensions of the fractional Poincaré-Sobolev inequality 
\begin{equation}\label{eq:FractionalPS}
\left(\avgint_Q |u-u_Q|^{p^*_{\delta}}dx\right)^{\frac{1}{p^*_{\delta}} }\leq c_{n,p^*_{\delta}}\,
 [u]_{W^{\delta,p}}(Q),  \qquad Q\in \cQ.
\end{equation}
%

%although the proofs of Theorem \ref{thm:PoSo-Aq-diam-delta} or are much more difficult, 
%they are extensions of the fractional Poincaré-Sobolev inequalities}
We are using here the notation $p_{\delta}^*$ for the fractional Sobolev exponent defined by $\frac{1}{p} -\frac{1}{ p_{\delta}^* }=\frac{\delta}{n}$ and also the notation
\begin{equation*}%\label{eq:LocalFractNotation}
[u]_{W^{\delta,p}(Q)}:=\ell(Q)^{\delta} \left(\avgint_Q\int_Q \frac{|u(x)-u(y)|^p}{|x-y|^{n+\delta p}}\, dy\, dx\right)^{1/p}.
\end{equation*}
To illustrate the self-improving philosophy in the case of cubes, we will discuss the main steps of the proof of \eqref{eq:FractionalPS} and a further improvement. This approach avoids completely the use of classical potential theory (see \cite{KLV} in the classical case and \cite{NPV} in the fractional case) and this is the reason why we can go beyond and consider other geometries and degeneracies.  Moreover, due to the presence of a more complicated geometry and to the presence of degeneracies produced by weights, the proofs of the results in the present paper require substantial work and involve serious technical difficulties.

We begin with the following easy to get 
inequality 
\begin{equation}\label{FPI-rough}
\avgint_Q \lvert u(x)-u_Q\rvert\,dx \leq  c_n\, [u]_{W^{\delta,p}(Q)} \qquad Q\in \cQ, \quad\delta>0
\end{equation}
which we claim encodes a lot of information, at least in the case $\delta \in (0,1)$. Indeed, it can be shown using the ideas in \cite{FPW98}, with the recently obtained  improvements from \cite{CantoPerez}, the following weak type version of \eqref{eq:FractionalPS}
\begin{equation}\label{weaktypeSobolev}
\|u-u_Q\|_{L^{p^*_{\delta},\infty}\big( Q, \frac{dx}{|Q|}\big)} \leq c_n \, p^*_{\delta}\, \, [u]_{W^{\delta,p}(Q)}, 
\end{equation}
which holds whenever $p\in [1,\infty)$ and for $\frac{1}{p} -\frac{1}{ p_{\delta}^* }=\frac{\delta}{n}$. Inequality \eqref{weaktypeSobolev}  really follows from 
\begin{equation*}%\label{eq:Dp}
\left( \sum_i a(Q_i)^{p^*_{\delta}} \frac{|Q_i|}{|Q|}  \right)^{\frac1{p^*_{\delta}}} \leq a(Q), 
\end{equation*}
where $a(Q)=[u]_{W^{\delta,p}(Q)}$ for any $\delta \in(0,1)$ and any $p\in [1,\infty)$ (observe that the constant in front is equal to 1). Now, since the truncation argument works for the functional $[u]_{W^{\delta,p}(Q)}$ as shown in \cite{DIV},  we can replace the weak norm \eqref{weaktypeSobolev} by the strong norm 
$$
\|u-u_Q\|_{L^{p^*_{\delta}}\big( Q, \frac{dx}{|Q|}\big)} \leq c_n \, p^*_{\delta}\, \, [u]_{W^{\delta,p}(Q)}
$$
(actually we can put the Lorentz  norm $\|\|_{L^{p^*_{\delta},p}}$ as well).  On the other hand, it turns out that this result is far from being optimal. Indeed, motivated by the works \cite{BBM, MS}, M. Milman obtained in \cite{M}, using ideas from interpolation theory, a very general way of proving the following self-improvement of \eqref{FPI-rough}, 
\begin{equation*}
\avgint_Q |u-u_Q|\,dx\leq c_{n}\,(1-\delta)^{\frac1p}\, \,[u]_{W^{\delta,p}(Q)},
\end{equation*}
where the highly  interesting extra gain $(1-\delta)^{\frac1p}$ appears in front. 
Then, exactly as above, the results in \cite{CantoPerez} combined with the truncation from \cite{DIV} yield the following highly interesting  improvement of \eqref{eq:FractionalPS}.
\begin{theorem} \label{thm:FracSobGain} Let $0<\delta<1$ and $1\leq p<\infty$. Then, 
\begin{equation*}%\label{e.Bgain}
\left(\avgint_Q |u(x)-u_Q|^{p^*_{\delta}}dx\right)^{\frac{1}{p^*_{\delta}} }\leq c_n \, p^*_{\delta}\, (1-\delta)^{\frac1p}\, [u]_{W^{\delta,p}(Q)}   \qquad Q\in \cQ
\end{equation*}
\end{theorem}

\begin{remark}\label{globalresult}
Standard arguments can be used to obtain the strong global estimate in $\mathbb{R}^{n}$ with the correction factor in front $(1-\delta)^{\frac1p}$.
\end{remark}

Some extensions of \eqref{eq:FractionalPS} to the case of domains can be found in \cite{HV}. 
We also refer to \cite{CDM} for extensions to the doubling  metric case.

We plan to extend this result into the context of the product space $\ccR$  from which we will derive the corresponding version
of \eqref{eq:FractionalPS} and then derive its corresponding Poincaré-Sobolev extension from our general method (see Theorem \ref{thm:AutomejorastrongBBM-A1-ccR}).

\section{Main results}

For the sake of clarity on the presentation, we will present our results for the basis $\cR$ and $\ccR$ separately.

\subsection{Analysis in \texorpdfstring{$\cR$}{cR}} \label{sec:analysis in cR} \

In this section we will study the self-improving phenomena in the context of $\cR$ wich will be served as a model for the case $\ccR$ that will be considered in Section \ref{sec:analysis in ccR}. First we will obtain Poincar\'e-Sobolev  type inequalities  involving higher order derivatives for rectangles in $\cR$ (some references about higher order Poincaré-Sobolev  type inequalities can be found in  \cite{AH,T}). To see where is the difficulty here, note in Definition \ref{def:smallness} that the condition on the family of rectangles is related to their measure. The classical one parameter theory for cubes works fine here, since the $n$-th power of the diameter (or the sidelength) of a cube is comparable to its measure. However, in the multiparametric setting of rectangles, this is no longer true. That is why the notion  of eccentricity comes into play.

 Our first main theorem on self-improving functionals is the following. We use here an oscillation with respect to the polynomial $P_R$ defined in Section \ref{sec:polynomials}, equation \eqref{eq:PR}. We anticipate here that this object $P_R$ is a projection to a space of polynomials of certain degree, and is the natural substitute for the average $f_R$ when dealing with functionals involving higher derivatives. The weighted estimate will be valid for any weight in the Muckenhoupt  class $w\in A_\infty:=\bigcup_{p>1}A_p$ associated to the basis $\cR$ (this can be defined for any basis, see more details about this class in Section \ref{sec:weights}).

\begin{theorem}\label{thm:AutomejorastrongcR}

Let $w$ be any $A_{\infty,\cR}$ weight in $\mathbb{R}^n$.
Consider also the functional   $a$ such that for some $p\ge 1$ it satisfies the weighted condition $SD^s_{p,\cR}(w)$ from \eqref{eq:SDp} with $s>0$ and constant $\|a\|$. 
Let $f$ be a locally integrable function such that
\begin{equation}\label{eq:UnWeightedStartingPointL1}
\frac{1}{|R|}\int_{R} |f-P_{R}f| \le a(R) \qquad R \in \cR.
\end{equation}
Then, there exists a dimensional constant $c_n$ such that for any $R \in \cR$, 
\begin{equation}\label{eq:First main estimate}
\left( \frac{1}{ w(R)  } \int_{ R }   |f -P_{R}f|^p     \,wdx\right)^{\frac{1}p}   \leq  
 c_n (1+s)\max\{\|a\|^s,1\}\,a(R).
\end{equation}
\end{theorem}

This result generalizes and improves on the corresponding result from \cite{PR-Poincare}  Theorem 1.24 where the linear constant $s$ is of exponential type.

As consequences of this sort of general template for a self-improving theorem, we will obtain Poincar\'e-Sobolev inequalities for higher order derivatives with a rather precise control on the $A_{p,\cR}$-like quantities involved. 

\begin{remark}
We will often (but not always) use a compact notation for the weighted local $L^q$ average over a rectangle $R$ defined as follows:
\begin{equation*}
\left\|f\right \|_{L^q\left (\frac{w\, dx}{w(R)}\right )}:=
 \left (\frac{1}{w(R)}\int_R |f |^q w dx\right )^{\frac{1}{q}}.
\end{equation*}

Similarly, we will use the standard notation for the normalized weak $(r,\infty)$ (quasi)norm: for any $0<r<\infty$, measurable $E$ and weight $w$, we define
\begin{equation*}
\| f \|_{L^{r,\infty}\big(\frac{w\,dx}{w(E)}\big)}:= \sup_{t>0}  t \, \left( \frac{1}{w(E)}w(\{x\in E:  |f(x)|>t\}) \right )^{1/r}.
\end{equation*}

\end{remark}

The general Theorem \ref{thm:AutomejorastrongcR} will be  used to derive several corollaries. One of the most important consequences is that we will get Poincar\'e-Sobolev type inequalities involving higher derivatives instead of the gradient.

We also present here the following result, which can be seen as a weak norm version of the main theorem, since we are asking the functional to only satisfy the ``rough" $D_p$ condition \eqref{eq:Dp} but obtaining a self-improving for the weak norm. It is a  self-improving result as in \cite{FPW98} but combined with the recent improvement obtained in \cite{CantoPerez} where linear bounds in both $[w]_{A_\infty,\cR}$ and $p$ are obtained instead of exponential.

\begin{theorem}\label{thm:AutomejoraweakcR}
Let $w$ be any $A_{\infty,\cR}$ weight in $\mathbb{R}^n$ and $a\in D_{p,\cR}(w)$. Let $f$ be a locally integrable function such that,
\begin{equation*}\label{eq:initialHypcR}
\avgint_{R} |f - P_{R}f|\le a(R) \qquad  R\in \cR.
\end{equation*}
Then there exists a dimensional constant $c>0$ such that for every $R\in \cR$,
\begin{equation*}
\big \| f-P_Rf \big \|_{L^{p,\infty}\big( R, \frac{w\,dx}{w(R)}\big)} \leq  c\, p\, [w]_{A_{\infty,\cR}}\,\|a\|_{D_{p,\cR}(w)} \, a(R).  
\end{equation*}
\end{theorem}

This result is completely new and has a lot of potential applications. We will show some of them in next sections.

Recall that
\begin{equation*}%\label{eq:Kolmog}
\big \| g \big\|_{L^{q}(X)} \leq \left( \frac{p}{p-q} \right)^\frac{{1}}{q} \, \big\|g\big\|_{L^{p,\infty}(X)} 
\end{equation*}
whenever $(X,\mu)$ is a probability space, and $0<q<p<\infty$ (sometimes called Kolmogorov's inequality.) Then, under the assumption 
of Theorem \ref{thm:AutomejoraweakcR}, if $q<p$, 
\begin{equation*}%\label{eq:conseq.weak-result}
\big \| f-P_Rf \big \|_{L^{q}\big( R, \frac{w\,dx}{w(R)}\big)} \leq  c_{p,q}\, [w]_{A_{\infty,\cR}}\,\|a\|_{D_{p,\cR}(w)} \, a(R).  
\end{equation*}

\subsection{Higher order \texorpdfstring{Poincar\'e}{Poincare}}
As a consequence of our general result we will prove, using condition \eqref{eq:SDp} as the key point, inequalities 
for $R\in \cR$ of the form
\begin{equation*}
\left (\frac{1}{w(R)}\int_R |f- P_Rf |^q w dx\right )^{\frac{1}{q}}\le C_w\,d(R)^m \left (\frac{1}{w(R)}\int_R |\nabla^m f|^{p}  w dx\right )^{\frac{1}{p}}  
\end{equation*}
where $C_w$ is a constant and $P_Rf$ is an appropriate optimal polynomial of degree less than $m$, defined in Section \ref{sec:polynomials}. The goal is to find the best possible improvement in the exponent $q$ on the left hand side and precise estimates on $C_w$. We will refer to the above inequality as a higher order  $(q,p)$-Poincar\'e-Sobolev weighted inequality.

A $(p,p)$ version of the above inequality in the case of cubes with $p\ge 1$, was proved in \cite{PR-Poincare}. Here we improve on that result achieving a higher Sobolev-type exponent on the left hand side and also valid for the family of rectangles $\cR$. As usual, when proving weighted inequalities related to rectangles, the appropriate class of weights is the \emph{strong} class defined analogously to the standard $A_{p}$ class as follows. We will say that $w \in A_{p,\cR}$ if
\begin{equation*}\label{eq:Ap-strong}
[w]_{A_{p,\cR}}:=\sup_{R\in \cR} \left( \frac{1}{|R|}\int_R w(x)\,dx\right) \left(\frac{1}{|R|}\int_R w(x)^{-\frac{1}{p-1}}\,dx \right)^{p-1}< \infty. 
\end{equation*}
The strong $A_{\infty,\cR}$ class is defined in the same way as in the cubic case and it enjoys the same geometric conditions (see Section \ref{sec:weights} for the details) and the corresponding constant $[w]_{A_{\infty,\cR}}$ is defined in \eqref{eq:AinftycR}.

The following corollary of a $(p,p)$-Poincar\'e inequality (for cubes) was already obtained as a corollary of \cite[Theorem 1.24]{PR-Poincare} which is a weaker version of our Theorem \ref{thm:AutomejorastrongcR} for the case of cubes.
\begin{corollary}[\cite{PR-Poincare}]\label{cor:Poincare(pp)-higher}
Let $p\geq1$ and let $w\in  A_{p}$  in $\mathbb{R}^n$.   Then the following $(p,p)$-Poincar\'e inequality holds
$$
\left (\frac{1}{w(Q)}\int_Q |f- P_Qf |^p\, w\ dx\right )^{\frac{1}{p}}\leq C\,[w]^{\frac{1}{p}}_{A_p}\ell(Q)^m 
\left (\frac{1}{w(Q)}\int_Q |\nabla^m f|^{p} \, w \ dx\right )^{\frac{1}{p}},
$$
where $C=C_{n,m}$ is a structural constant and $\ell(Q)$ is the sidelength of the cube $Q  \subset \mathbb{R}^n$.
\end{corollary}

The method of proof goes beyond Poincar\'e to reach Poincar\'e-Sobolev type inequalities with a higher exponent on the left hand side. In the following theorem the method seems to work best when the weight adds some degeneracy, which is a bit surprising. More precisely, we will say that a weight $w\in A_{\infty,\cR}$ is \textbf{nontrivial} whenever $[w]_{A_{\infty,\cR}}>1$. The same notion for other bases of rectangles or cubes will be used as well.

\begin{theorem} \label{thm:PoSo-Aq-diam-delta}
Let $1 \leq p < n $ and let $w\in A_{q,\cR}$ in $\mathbb{R}^n$ be a {\bf nontrivial weight} with $1\le q\le p$. Let also $p_w^* $ be defined by 
\begin{equation}\label{eq:ptimes-Aq}
\frac{1}{p} -\frac{1}{ p_w^* }=\frac{\delta}{n}\frac{1}{q+\log [w]_{A_{q,\cR}}}.
\end{equation}
For a fixed $\delta>0$, let $a$ be the functional defined by
\begin{equation}\label{eq:model-d(Q)^delta-a(Q)} 
a(R)=d(R)^\delta\left(\frac{1}{w(R)}\mu(R) \right)^{1/p},
\end{equation}
where $\mu $ is any Radon measure in $\mathbb{R}^n$. Suppose that  $f$ satisfies 
\begin{equation}\label{eq:Sarting-PoSo-Aq}
\frac{1}{|R|}\int_{R} |f-P_{R}f| \le a(R)
\end{equation}
for every rectangle  $R \subset \mathbb{R}^n$. Then, there exists a dimensional constant $C_n$ such that for any rectangle $R \subset \mathbb{R}^n$
\begin{equation*}
\left( \frac{1}{ w(R)  } \int_{ R }   |f -P_{R}f|^{p_w^*}     \,wdx\right)^{\frac{1}{p_w^*}}  	\leq B_{w,q}\, C_n \frac{1}{\delta} a(R), 
\end{equation*}
where $B_{w,q}=\frac{1+ \log[w]_{A_{q,\cR}}^\frac{1}{q}}{ \log[w]_{A_{q,\cR}}^\frac{1}{q}}$.

\end{theorem}
Note that the dependence on the weight can be disregarded if we are looking for the asymptotic behavior when $[w]_{{A_{q,\cR}}}\to\infty$ (we can bound that expression by 2 by considering weights such that $[w]_{{A_{q,\cR}}}\ge e^q$). The real problem is with \emph{flat weights}. We borrow that terminology from \cite{ParissisRela} to refer to weights with small $A_\infty$ constant, close to 1. Those weights are, in some sense, getting closer to being just a multiple of the Lebesgue measure, and that is why this situation is counter intuitive, since one would expect an easier problem for this latter case. Also note that there is some sort of balance between the exponent $p^*_w$ and the constant in the inequality. The closer the $A_q$ constant gets to 1, the closer the value of $p^*_w$ gets to the optimal Sobolev exponent. But it should also be observed that the constant in the inequality blows up when $[w]_{A_{q,\cR}}\to 1$. And that is precisely the reason why this result can not be specialized on the limit case of the constant weight, but we can give a reasonable substitute by using weak norms.

\begin{theorem} \label{thm:PoSo-Aq-diam-delta-weak}
Let $1 \leq p < n $ and let $w\in A_{q,\cR}$ in $\mathbb{R}^n$  with $1\le q\le p$. If we define the exponent $p^*_w$ as in \eqref{eq:ptimes-Aq}, the functional $a$ as in  \eqref{eq:model-d(Q)^delta-a(Q)} and consider a function satisfying \eqref{eq:Sarting-PoSo-Aq}, there exists a dimensional constant $c_{n}$ such that for any rectangle $R \subset \mathbb{R}^n$
\begin{equation}\label{eq:PoSo-Aq-diam-delta-weak}
\left \|f -P_{R}f\right \|_{L^{p_w^*,\infty}(R,\frac{w dx}{w(R)})}	\leq \, c_n\, \frac{1}{\delta} a(R).
\end{equation}
\end{theorem}

The idea is to use a specific choice of the functional $a$ as the starting point. Since a $(1,1)$-Poincar\'e inequality involving higher derivatives is known to hold on rectangles, we obtain the following result.

\begin{corollary}\label{cor:ptimes-Aq-m-derivatives}
Let $1 \le p < n $ and let $w\in A_{q,\cR}$ in $\mathbb{R}^n$ with $1\le q\le p$. Let $p_w^*$ as in the previous theorem. Then the following inequality holds, 
\begin{equation*}\label{eq:ptimes-Aq-gradient}
\left\|f -P_{R}f \right \|_{L^{p_w^*}\left (R,\frac{w\, dx}{w(R)}\right )} 
\leq B_{w,q} C_{n}\frac{1}{m} [w]_{A_{q,\cR}}^{\frac{1}{p}}
 \,d(R)^m
\left\|\nabla^m f \right \|_{L^{p}\left (R,\frac{w\, dx}{w(R)}\right )} 
\end{equation*}
for every rectangle $R\in \cR$, where as before, $B_{w,q}=\frac{1+ \log[w]_{A_{q,\cR}}^\frac{1}{q}}{ \log[w]_{A_{q,\cR}}^\frac{1}{q}}$.

We also have the estimate for the weak norm
\begin{equation}\label{eq:weaktype-Aq-gradient}
\left \|f -P_{R}f\right \|_{L^{p_w^*,\infty}(R,\frac{w dx}{w(R)})}	\leq \, C_{n,p} \frac{1}{m} [w]_{A_{q,\cR}}^{\frac{1}{p}}
 \,d(R)^m
\left\|\nabla^m f \right \|_{L^{p}\left (R,\frac{w\, dx}{w(R)}\right )}. 
\end{equation}
\end{corollary}

The proof of this result can be found in Section \ref{sec:proofscR}.

\begin{remark} We remark 
that for any nontrivial weight separated from the constant weight, say for instance
$[w]_{{A_{q,\cR}}}\ge e^q$, we can apply Theorem \ref{thm:PoSo-Aq-diam-delta} to an appropriate starting point. Without this non-degeneracy condition, and in particular in the case of flat weights, we avoid the logarithmic blowup but at the cost of obtaining only a weak norm estimate. Of course, when $m=1$ the truncation  argument can be carried out to reach the strong norm from \eqref{eq:weaktype-Aq-gradient}, to wit
\begin{equation*}
\left \|f -P_{R}f\right \|_{L^{p_w^*}(R,\frac{w dx}{w(R)})}	\leq \, c_{n}  [w]_{A_{q,\cR}}^{\frac{1}{p}}
 \,d(R)
\left\|\nabla f \right \|_{L^{p}\left (R,\frac{w\, dx}{w(R)}\right )}. 
\end{equation*}

\end{remark}

\subsection{Fractional \texorpdfstring{Poincar\'e}{Poincare}-Sobolev and eccentricity.}
This section is motivated by the fractional Poincaré-Sobolev inequalities \eqref{eq:FractionalPS}. First we will introduce a fractional type Poincaré estimate adapted to the  product space $\cR$  from which we will derive the  corresponding Poincaré-Sobolev type result. 

As before, we will introduce functionals satisfying the smallness preservation condition $SD_{p,\cR}^s$ for some choice of parameters 
where the concept of eccentricity will play a central role in the proofs.   
It comes from the natural interplay between the notion of diameter and measure for a given rectangle.

\begin{definition} 

Let $R\in \cR$  be any rectangle in $\mathbb{R}^{n}$.  We define the eccentricity  as the number 
\begin{equation}\label{eq:eccentricity}
e(R):=\frac{|R|^\frac{1}{n}}{d(R)}. 
\end{equation}
\end{definition}

The main property we use is the following observation:   all the involved families of rectangles will be obtained from dyadic partitions of a fixed initial rectangle. Therefore, for each instance of this initial choice of a rectangle, all of them are dyadic children of a given rectangle. The main feature here is that all of them have the same eccentricity.

\begin{lemma}\label{lem:eccentr}
Let $\tilde{R}$ be any dyadic descendant of $R \in \cR$. Then 
$$e(\tilde{R})=e(R).$$
\end{lemma}

Indeed, if $j$ be the dyadic level to which $\tilde{R}$ belongs. Then,
\begin{equation}\label{eq:eccentricity-children}
e(\tilde{R})=\frac{|\tilde{R}|^{1/n}}{d(\tilde{R})}=\frac{(|R|2^{-jn})^{1/n}}{d(R)2^{-j}}=e(R). 
\end{equation}

By introducing this notion of eccentricity, we now are able to formulate other corollaries from our general self-improving Theorem \ref{thm:PoSo-Aq-diam-delta}. We have the following  intermediate estimate.

\begin{corollary}\label{cor:fractional}
Let $w\in A_{q,\cR}$ in $\mathbb{R}^{n}$ with $1\le q\le p<n$. Let $p_w^*$ as in \eqref{eq:ptimes-Aq}, namely
\begin{equation*}
\frac{1}{p} -\frac{1}{ p_w^* }=\frac{\delta}{n}\frac{1}{q+\log [w]_{A_{q,\cR}}}.
\end{equation*}
Let $R\in \cR$ and 
$\delta>0$. Consider the function
\begin{equation}\label{eq:A(R,x)}
A(R,x)= \int_R \frac{|f(x)-f(y)|^p}{|x-y|^{n+\delta p}} \,dy, \qquad  \qquad x\in R.
\end{equation}
Then the following inequality holds, %
\begin{equation}\label{eq:fractional}
\left \|f-f_R \right \|_{L^{p_w^*}(R, \frac{wdx}{w(R)})}\leq C_{n,p,q}
\frac{1}{\delta}\, 
[w]^\frac{1}{p}_{A_{p,\cR}}\,
\frac{d(R)^{\delta}}{e(R)^{\frac{n}{p}}}\left( \frac{1}{w(R)} \int_R A(R,\cdot)w\,dx\right)^\frac{1}{p}.
\end{equation}

\end{corollary}

\begin{remark}
We will prove this result from an initial inequality \eqref{eq:general-starting-a(R)} with a non standard functional $a(R)$ involving the quantity $A(R,x)$, namely
\begin{equation*}
\frac{1}{|R|}\int_R |f-f_R|\leq [w]^\frac{1}{p}_{A_{p,\cR}}\,\frac{d(R)^{\delta}}{e(R)^\frac{n}{p}}\left(\frac{1}{w(R)}\int_R A(R,x)\,w(x)dx \right)^\frac{1}{p}.
\end{equation*}
\end{remark}

\begin{remark}
There should be here a result as in Theorem \ref{thm: A1PSFractBBM} with the extra factor  $(1-\delta)^{\frac1p}$ at least in the case $w\in A_{1,\cR}$. We don't know how to do it. However, we proved the corresponding variation of this result in the context of $\ccR$ in Theorem \ref{thm:AutomejorastrongBBM-A1-ccR}.
\end{remark}

Note that,   on the one hand  the above starting point  does not involve any derivative, but on the other hand  it is imposing a  somewhat stronger control over the $L^1$ oscillation, since it is known that for any $0<\delta<1$, and any $p\geq 1$ one has
\begin{equation*}
\ell(Q)^\delta \left(\avgint_Q\int_Q\frac{|f(x)-f(y)|^p}{|x-y|^{n+\delta p}}dydx\right)^{1/p} \lesssim C_{\delta,p,n} \ell(Q)\left(\avgint_Q |\nabla f|^p\right)^{1/p}.
\end{equation*}  
We will in fact provide the proof of an improved version of this inequality for the case $p=1$ in Lemma \ref{lem:oneparameterFractNabla}, inequality \eqref{eq:roughfractionalPI}. The proof is included in
the Appendix \ref{sec:App:oneparameterFractNabla}.

\subsection{A weaker starting point} 
 
The aim of this section is to show that we can replace the $L^1$ mean oscillation in \eqref{eq:UnWeightedStartingPointL1} as initial hypothesis by a weaker condition to derive the same results. Indeed, as shown in Theorems \ref{thm:AutomejorastrongcR} and Theorem \ref{thm:PoSo-Aq-diam-delta},  the natural starting point is given by the expression
\begin{equation}\label{eq:L1-starting}
\avgint_R |f(x)-f_R|\,dx \le a(R)
\end{equation}
where $a$ is a functional satisfying some sort of discrete summation condition. We show in next theorem that we can replace the $L^1$ mean oscillation in \eqref{eq:L1-starting} 
by $L^{\delta}$, $0<\delta<1$ weakening the initial assumption. This idea was already considered in 
\cite{LuPerez02} and \cite{LernerPerez-poin-scandi} but the results and the methods are not so precise as the ones we obtain here within a different  context.

\begin{theorem}\label{thm:delta-StartingPoint}   
Let  $f$ be a measurable function and let $\delta\in (0,1)$. Suppose that there is a functional $a$ such that,  
\begin{equation}\label{eq:delta-StartingPoint}
\inf_{c\in \mathbb R}\left (\avgint_R |f(x)-c|^\delta\right )^{\frac{1}{\delta}} dx \le a(R) \qquad  R\in \cR.
\end{equation}
If the functional $a$ satisfies the $D_{p,\cR}$  condition for some $p>1$ or  the $SD_{1,\cR}^s$ condition for some $s>0$, then there is a self-improving to $L^1$, namely 
\begin{equation}\label{eq:L^1-goal}
\inf_{c\in \mathbb R}\avgint_R |f(x)-c| dx \leq C\, a(R), \qquad R\in \cR,
\end{equation}
where $C$ depends on $p,\delta,s,\|a\|_{D_{p,\cR} }$ or \, $\|a\|_{SD^s_{1,\cR}}$.
 \end{theorem}

\begin{remark}
This is also very interesting even in the case of $BMO$. In that case $a(R)=1$ and the question can be phrased in terms of minimal conditions on the starting estimate \eqref{eq:delta-StartingPoint} to conclude the membership to $BMO$. An approach to this problem (in the context of cubes) using Bellman functions can be found in \cite{LSSVZ-BMO}. However, this result has been improved in several ways in the more recent work \cite{CPR}. It may be possible to adapt these ideas to the context of the current paper. 
\end{remark}

\subsection{Analysis in \texorpdfstring{$\ccR$}{ccR}} \label{sec:analysis in ccR}\

We will show that we can apply our techniques to obtain other type of results for the class $\ccR$ which denotes the family formed by rectangles $R$ of the form $R=I_1\times I_2 \subset \mathbb{R}^{n}$ where $I_1\subset \mathbb{R}^{n_1}$ and $I_2\subset \mathbb{R}^{n_2}$ are cubes with sides parallel to the coordinate axes.  We will discuss different possibilities for the starting inequality \eqref{eq:general-starting-a(R)} 
which will lead to different results on Poincar\'e-Sobolev  type inequalities more in the spirit of 
\cite{ShiTor, LuWheeden}.

The first results in the product case context were obtained by 
X. L. Shi and A. Torchinsky in \cite{ShiTor}. They obtained $(q,p)$-Poincaré-Sobolev  inequalities  with weights satisfying certain strong conditions.  Later on, Seng-Kee Chua \cite{Ch1995} obtained some results assuming mixed norm assumptions on the gradient which we will not consider. Later on, G. Lu and R. Wheeden also derived in \cite{LuWheeden} Poincaré-Sobolev 
inequalities in the context of vector fields. We improve these results in the classical context. This case could be certainly treated as a particular case of the general theory developed for the family $\cR$, since the selfimproving method works in the same way. But here we will consider different starting points, better adapted to this particular geometry.

We will use the following notation: for a given function  $f:U\to\mathbb{R}$ defined on the open set $U\subset \mathbb{R}^{n_1} \times \mathbb{R}^{n_2}$, we will write $f(x)=f(x_1,x_2)$ where $x_1$ stands for the first $n_1$ variables and $x_2$ stands for the remaining $n_2$ variables. 
$\nabla_1f$ will denote the partial gradient of $f$ containing the $x_1$-derivatives and 
similarly  $\nabla_2f$ will denote the partial gradient of $f$ containing the $x_2$-derivatives. 
Recall that the family $\ccR$ is defined by rectangles of the form $R=I_1\times I_2$ where in $I_1\subset \mathbb{R}^{n_1}$ and $I_2 \subset \mathbb{R}^{n_2}$ are cubes with sides parallel to the coordinate axes and $n:=n_1+n_2$.  As already mentioned, the classical method to derive these type of results is based on  finding a representation formula and this is the way was done in the first paper in this context \cite{ShiTor}.   We follow a completely different path using some of the key ideas from \cite{FPW98}.

\begin{theorem}\label{thm:AutomejorastrongccR}  \, Let $w\in A_{q,\ccR}$ in $\mathbb{R}^{n}$ and $1 \leq q \leq p.$
Let also 
$$\frac1p-\frac1{p^*}= \frac1{n}\frac1{(q+\log [w]_{A_{q,\ccR}} )}.
$$
Then, there exists a constant $c=c(n,p,q)>0$ such that for every Lipschitz function $f$ and any $R=I_1\times I_2\in \ccR$, 
\begin{equation*}
\left\|f-f_R\right \|_{L^{p^*}(R, \frac{wdx}{w(R)})}
\leq c\,
[w]_{A_{p,\ccR}}^{\frac{1}{p}} \left (a_1(R)+a_2(R)
 \right ),
\end{equation*}
where 
\begin{equation*}
a_1(R)= \ell(I_1) 
\left\|\nabla_1 f\right \|_{L^{p}(R, \frac{wdx}{w(R)})}
\quad \text{ and }\quad 
a_2(R)= \ell(I_2) 
\left\|\nabla_2 f\right \|_{L^{p}(R, \frac{wdx}{w(R)})}.
\end{equation*}

\end{theorem}

This result will be proved in Section \ref{sec:Bi-parameterPSI}.
As far as we know these results are not known in the literature (for instance \cite{Ch1995, LuWheeden}) even for the case $p>1$ which is usually simpler.
We will show that this theorem 
%\eqref{eq:Bi-parameterPoincareSobolev}  
will follow essentially from the following $(1,1)$-Poincaré inequality in product spaces 
which seems to be unknown. 
\begin{lemma}\label{lem:(1,1)PI-ccR} 
There exists a dimensional constant $c>0$ such that for every Lipschitz function 
$f:\mathbb{R}^{n_1} \times \mathbb{R}^{n_2}\to \mathbb{R}$ and for any $R=I_1\times I_2\in \ccR$,
\begin{equation} \label{eq:PIsidelenght}
\avgint_{R} |f - f_{R}|\leq c\, \ell(I_1)\avgint_{R} |\nabla_1 f| + c\,\ell(I_2)\avgint_{R} |\nabla_2 f|.
\end{equation}

\end{lemma}

This lemma is inspired by the work of Shi and Torchinsky \cite{ShiTor} but they do not consider the case $p=1$ since their method fails. 
Inequality \eqref{eq:PIsidelenght} will play a similar role as  \eqref{eq:Poincare-(1,1)}. 
%although the appearance of the sidelengths $\ell(I_i)$ is not as efficient as the diameter $d(R)$. 
However, there are two different points which makes the initial estimate \eqref{eq:PIsidelenght} different to the one we have been considering so far. First is that we use rectangles in the family $\ccR$ and second that the right hand side is formed by the sum of two functionals. 

We observe that if we enlarge \eqref{eq:PIsidelenght} replacing both sidelenghts,  $\ell(I_1)$ and  $\ell(I_2)$, by the larger quantity $d(R)$,  then the new functional is essentially the right hand side of \eqref{eq:Poincare-(1,1)} with $\Omega=R$ since $|\nabla_1 f| +|\nabla_2 f| \approx |\nabla f|$.

We will prove Lemma \ref{lem:(1,1)PI-ccR} in Section \ref{sec:Prelim-ccR} in two ways. The first proof will follow from Proposition \ref{pro:LW} as a consequence of a point-wise estimate obtained by Lu and Wheeden in \cite{LuWheeden}. Additionally, we will provide a second proof as a consequence of a ``fractional" Poincaré type inequality proven in Proposition \ref{pro:1-1FPIproductSpaces}. This proposition will be further improved in the following theorem which can be seen as a biparametric version of the results 
from \cite{BBM} and improved by M. Milman in \cite{M} 
(see also \cite{MS} and the recent interesting works \cite{BVY} \cite{DD} \cite{DM}).

\begin{theorem}\label{thm:BBMbiparametrico} 
Let $p_1,p_2\in [1,\infty)$ and let $\delta_1,\delta_2 \in(0,1).$  Then there exist dimensional constants $c_{n_1}, c_{n_2}>0$ such that  
for any  rectangle $R=I_1\times I_2 \in \ccR$
\begin{equation*}
\avgint_{R} |f - f_{R}|\leq c_{n_1} (1-\delta_1)^{\frac{1}{p_1}} a_1(R) 
+ \, c_{n_2} \, (1-\delta_2)^{\frac{1}{p_2}}a_2(R),
\end{equation*}
where 
\begin{equation*}
a_1(R)= \ell(I_1)^{\delta_1}\left(\avgint_{R}\int_{ I_1}  \frac{|f(x_1,x_2)- f(y_1,x_2)|^{p_1}}{|x_1-y_1|^{n_1+p_1\delta_1}}\, dy_1 dx_1\  dx_2\right)^{\frac1{p_1}}
\end{equation*}
and 
\begin{equation*}
a_2(R)=  \ell(I_2)^{\delta_2}\left(\avgint_{R}  \int_{I_2} \frac{|f(y_1, x_2) - f(y_1, y_2) |^{p_2}}{|x_2-y_2|^{n_2+p_2\delta_2}}\ dy_1 dx_2 dy_2.\right)^{\frac{1}{p_2}}
\end{equation*}
\end{theorem} 

Once we have Theorem \ref{thm:BBMbiparametrico} proved, it is immediate to obtain Lemma \ref{lem:(1,1)PI-ccR} from the case $p_1=p_2=1$.

This result will be the starting point to derive a bi-parametric version of the following  (one-parameter) result which can be found in \cite{HMPV}.

\begin{theorem}\label{thm: A1PSFractBBM} (\cite{HMPV})
Let $\delta \in (0,1)$ and  $w \in A_1 $. Also let $p_{\delta,w}^*$ be the weighted fractional Sobolev exponent defined by
\begin{equation}\label{DegSobExpCubes}
\frac{1}{p} -\frac{1}{ p_{\delta,w}^* }=   \frac{\delta}{n} \, \frac{1}{1+\log [w]_{A_1}}
\end{equation}
There is a constant $C_{n,p}>0$ such that for every cube $Q$ in $\mathbb{R}^n$, 
\begin{equation*}
\left(  \avgint_{ Q }   |u -u_{Q}|^{p_{\delta,w}^*}     \,wdx\right)^{\frac{1}{p_{\delta,w}^*}}  
\leq C_{n,p}\, (1-\delta)^{\frac1p} \,  [w]^{\frac{1}{p}}_{A_1}a(Q),
\end{equation*}
where
\begin{equation*}
a(Q)= \ell(Q)^{\delta} \,
\left(  \avgint_{Q}  \int_{Q}  \frac{|u(x)- u(y)|^p}{|x-y|^{n+p\delta}}\,dy\,wdx\right)^{\frac1p}.   
\end{equation*}
\end{theorem}

The bi-parametric counterpart of this result is the following.

\begin{theorem}\label{thm:AutomejorastrongBBM-A1-ccR}
Let $\delta \in (0,1)$ and  $w \in A_{1,\ccR} $. Also, let  
\begin{equation}\label{DegSobExp-ccR}
\frac{1}{p} -\frac{1}{ p_{\delta,w}^* }=   \frac{\delta}{n} \, \frac{1}{1+\log [w]_{A_{1,\ccR}} }
\end{equation}
Then there exists a  constant $c$ such that for each 
$R=I_1\times I_2 \in \ccR$ 
\begin{equation*}
\left(  \avgint_{ R }   |f -f_{R}|^{p_{\delta,w}^*}     \,wdx\right)^{\frac{1}{p_{\delta,w}^*}} \leq 
c\,[w]_{A_{1,\ccR}}^{\frac1p}\, (1-\delta)^{\frac1p}\Big(a_1(R)+a_2(R)\Big)
\end{equation*}
where
\begin{equation*}
a_1(R)=  
\left( \frac{\ell(I_1)^{p\delta}  }{w(R)} \int_{R} \int_{ I_1}  \frac{|f(x_1,x_2)- f(y_1,x_2)|^p}{|x_1-y_1|^{n_1+p\delta}}\, w(x_1,x_2) dx_1  dx_2\ dy_1\right)^{\frac1p}
\end{equation*}
and
\begin{equation*}
a_2(R)=  
\left( \frac{\ell(I_2)^{p\delta}  }{w(R)} \int_{R} \int_{ I_2}  \frac{|f(y_1,x_2)- f(y_1,y_2)|^p}{|x_2-y_2|^{n_2+p\delta}}\, w(y_1,y_2)\,  dy_1dy_2\ dx_2  \right)^{\frac1p}.
\end{equation*}

\end{theorem} 

This result is completely new and the proof  can be found in Section \ref{sec:Bi-parameterFractionalBBMPSI}.
An interesting question here is to extend this result beyond $A_{1,\ccR}$, for instance to the $A_{p,\ccR}$ case which is not known even in the one parameter situation, Theorem \ref{thm:AutomejorastrongBBM-A1-ccR}.%

\subsection{Outline}
The paper  is organized as follows. In Section \ref{sec:prelim} we will include some preliminary results needed for our proofs. In particular, we will show the geometric summability conditions satisfied by different functionals $a(R)$. In Section \ref{sec:proofscR} we will provide the proofs of the main results in the case of multiparameter rectangles in the family $\cR$. 
In Section \ref{sec:Prelim-ccR} we obtain the crucial starting $(1,1)$-Poincar\'e inequalities adapted to the basis $\ccR$.  In Section \ref{sec:Bi-parameterPSI}  we use this starting point as input for our self-improving method to obtain Poincar\'e-Sobolev type inequalities for the basis $\ccR$ of rectangles given by product of cubes. Following the same line of ideas, in Section \ref{sec:Bi-parameterFractionalBBMPSI} we obtain fractional Sobolev-Poincar\'e inequality for the biparameter case as presented in Theorem \ref{thm:AutomejorastrongBBM-A1-ccR} even with the extra gain $(1-\delta)^{\frac1p}$.  
Finally, in the Appendix \ref{sec:App-Truncation} we provide the proofs of the so called truncation method and in Appendix \ref{sec:App:oneparameterFractNabla} we give the proof of Lemma \ref{lem:oneparameterFractNabla}.

\subsection{Acknowledgement}  The last author is very grateful to Professors Oscar Domínguez and Mario Milman for enlighteling conversations about the results concerning fractional (uniparametric) Poincaré inequalities with extra gain like in Theorem  \ref{thm:BBM}. 
In particular  we ackonwledge the very interesting work \cite{M}.

\section{Preliminaries}\label{sec:prelim} 
We include here some well known preliminaries and some needed lemmas and further properties.
Throughout this section, the ambient space will be $\mathbb{R}^n$ and $R$ will denote an arbitrary rectangle in $\cR$, the 
family of rectangles defined as $n$-fold product of intervals on $\mathbb{R}$.

\subsection{Smallness preserving functionals}

 We start with some model examples of functionals mentioned in the statements of our results. As we will see, the constant eccentricity is crucial.

\begin{example} \label{ex:modelexample} \, Consider the functional defined by 
\begin{equation}\label{eq:Model-a(R)}
a(R)=d(R)^{\delta}\left(\frac{1}{w(R)}\int_R A(R,x)\,dx \right)^{1/p}
\end{equation}
where $A(R,x)$ is nonnegative and increasing in $R$. That is, $R_1 \subset R_2$ implies that  $A(R_1,x)\leq A(R_2,x)$.    We will verify that this functional satisfies Definition  \ref{def:smallness}, more precisely we will check that
$a\in SD^{n/\delta}_{p,\cR}(w)$ for a weight $w$  in $\mathbb{R}^n$. More importantly, that $\|a\|_{SD^{n/\delta}_{p,\cR}(w)}=1$ . Consider $q>1$ such that $\frac{\delta p q}{n}>1$. Then by the key eccentricity property in Lemma \ref{lem:eccentr}  %\label{lem:eccentr}
\begin{eqnarray*}
\sum_i a(R_i)^pw(R_i)&=&\sum_i d(R_i)^{\delta p} \int_{R_i}A(R_i,x)\,dx\\
&=&\frac{1}{e(R)^{\delta p}}\sum_i |R_i|^{p\delta/n}\int_{R_i}A(R_i,x)\,dx \\
& \leq & \frac{1}{e(R)^{\delta p}} \left (\sum_i  |R_i|^{\frac{\delta p q}{n}}\right )^{1/q}\left( \sum_i \left[ \int_{R_i} A(R_i,x)\,dx\right]^{q'} \right)^\frac{1}{q'} \\
&\leq& \frac{1}{e(R)^{\delta p}}\left( \sum_i |R_i|\right)^{\frac{\delta p}{n}}\int_R A(R,x)\,dx.
\end{eqnarray*}
Now, the eccentricity definition gives the desired bound,
\begin{eqnarray*}
\sum_i a(R_i)^p w(R_i)  &\leq& \frac{|R|^{\frac{\delta p}{n}}}{e(R)^{\delta p}} 
\left( \frac{|\bigcup_iR_i|}{|R|}  \right)^{\frac{\delta p}{n}}  \,\int_R A(R,x)\,dx\\
& =& d(R)^{\delta p}  \,  \left( \frac{|\bigcup_iR_i|}{|R|}  \right)^{\frac{\delta p}{n}} \int_R A(R,x)\,dx\\
& = &   \left( \frac{|\bigcup_iR_i|}{|R|}  \right)^{\frac{\delta p}{n}}\, a(R)^pw(R). 
\end{eqnarray*}

\end{example}

\begin{example}\label{ex:mainexample}
The example above belongs to a rather wide family of functionals given by averages of general measures:  $a(R)=d(R)^\delta\left( \frac{\mu(R)}{w(R)}\right)^{1/p}$, $\delta>0$. The previous computation shows that $a \in SD_{p,\cR}^{n/\delta}(w)$ for any weight $w$ in $\mathbb{R}^n.$ 
\end{example}

The above estimates are sufficient to obtain $(p,p)$-Poincar\'e inequalities since we have that $\|a\|_{SD^{n/\delta}_{p,\cR}(w)} \leq 1$. In order to reach higher exponents of integrability on the left hand side and get a Poincar\'e-Sobolev type inequality, we need to pay a price allowing $\|a\|_{SD_{p,\cR}^s(w)}$ to be larger than 1. A sort of natural way to do this is to consider the involved weight to be in a ``better" $A_{q,\cR}$ class. We follow the line of ideas from \cite{PR-Poincare} adapted to our setting of rectangles.

\begin{lemma}\label{lem:p^*}
Consider two indices $p,q$ such that $1\le q\le p < n$ and let $w \in A_{q,\cR}$ in $\mathbb{R}^n$.  For $\delta>0$, consider the functional 
$$
a(R)=d(R)^\delta\left(\frac{1}{w(R)}\mu(R) \right)^{1/p} \qquad R \in \cR.
$$
For $M>1$ we define $p_M^*:=p(n,q,M)$ by the condition
\begin{equation*}\label{eq:ptimesM}
\frac{1}{p} -\frac{1}{ p_M^*}=\frac{\delta}{nqM}.
\end{equation*}
Then $a$ satisfies Definition \ref{def:smallness} with parameters $p=p^*_{M}$ and $s=\frac{nM'}{\delta}$, namely if $R \in \cR$ and  $\{R_i\}_i$ is any family of pairwise disjoint dyadic subrectangles of $R$ the following inequality holds:
\begin{equation}\label{eq:Keyestimate}
\left( \sum_{i}a(R_i)^{ p^*_{M} }\,\frac{w(R_i)}{w(R)}\right)^{\frac{1}{p^*_{M}}}  \leq  [w]_{A_{q,\cR}}^{ \frac{ \delta}{nqM} }
\,\left (   \frac{|\bigcup_iR_i|}{|R|}      \right )^{ \frac{\delta}{nM'}  } a(R). 
\end{equation}
\end{lemma}
Note that $p_M^*$ is smaller than the classical Sobolev exponent, namely the sharp one corresponding to the Lebsegue measure case:
\begin{equation*}
p<p_M^*<p^*=\frac{pn}{n-p} \qquad 1\leq p<n.
\end{equation*}
This condition says that the functional $a$ ``preserves smallness'' for the exponent $p_M^*$ defined in \eqref{eq:ptimes-Aq} with index $s= \frac{nM'}{\delta}$ and constant 
\begin{equation*}
\|a\|_{SD_{p^*_M,\cR}^{ \frac{nM'}{\delta} }(w)} \leq [w]_{A_{q,\cR}}^{ \frac{ \delta}{nqM} }.
\end{equation*}
That is, $a\in SD_{p^*_M,\cR}^{\frac{nM'}{\delta}}(w)$. The notation here is somewhat involved, but we will use it in a more simplified manner in our applications.
Recall that, as in the case of standard cubic $A_q$ weights, the strong $A_{q, \cR}$ weights satisfy the geometric estimate
\begin{equation*}
\left (\frac{|E|}{|R|}  \right )^{q}\leq [w]_{A_{q,\cR}}  \frac{w(E)}{w(R)},
\end{equation*}
valid for any subset   $E\subset R$ and any $R\in \cR$. 
Now we can proceed to the proof of the main estimate for the functional $a$.

\begin{proof}[Proof of Lemma \ref{lem:p^*}]
 Let $M>1$. For simplicity in the exposition, we will omit the subindex $M$ and just use $p^*$ instead of $p^*_M$. To verify the smallness preservation for the functional $a$, we compute, using again the key eccentricity property in Lemma \ref{lem:eccentr},  
\begin{eqnarray*}
\sum_{i}a(R_i)^{p^*}w(R_i) & = & \sum_{i} \mu(R_i)^{ \frac{p^*}{p} }    \left (    \frac{d(R_i)^\delta}{w(R_i)^{\frac{1}{p} -\frac{1}{p^*}} } \right )^{ p^*  }          \\
&= & \frac{1}{e(R)^{\delta p^*}}\sum_{i} \mu(R_i)^{\frac{p^*}{p}}    \left(    \frac{ |R_i|^\frac{\delta}{n} }{w(R_i)^{ \frac{\delta}{qnM}  } } \right )^{ p^* } \\
&= & \frac{1}{e(R)^{\delta p^*}}\sum_{i} \mu(R_i)^{\frac{p^*}{p}}    \left(    \frac{ |R_i| }{w(R_i)^{ \frac{1}{qM}  } } \right )^{ \frac{\delta p^*}{n} } \\
& \leq &   \frac{[w]_{A_{q,\cR}}^{\frac{\delta p^*}{nqM} }}{e(R)^{\delta p^*}}
\left( \frac{|R|^{q}}{w(R) }   \right)  ^{ \frac{\delta p^*}{nqM}   }  \  
  \sum_{i} \mu(R_i)^{\frac{p^*}{p}}    |R_i|^{ \frac{\delta p^* }{ nM'}   }. \\
\end{eqnarray*}
Now consider any $t>1$ such that $t\frac{\delta p^* }{ nM'}\geq 1 $ and apply H\"older's inequality to the sum above to obtain 
\begin{eqnarray*}
  \sum_{i} \mu(R_i)^{\frac{p^*}{p}}    |R_i|^{ \frac{\delta p^* }{ nM'}   } & \le &   \left (\sum_i  \mu(R_i)^{  \frac{t'p^*}{p} }    \right )^\frac{1}{t'} \left (\sum_i |R_i|\right )^{ \frac{\delta p^* }{ nM'}  }\\
& \le &    \left (\sum_i  \mu(R_i)    \right )^{ \frac{p^*}{p} }  \,  
\left (  \frac{|\bigcup_iR_i|}{|R|}\right)^{\frac{\delta p^* }{ nM'} } |R|^\frac{\delta p^* }{ nM'} .\\
& \leq &    \mu(R)^{ \frac{p^*}{p} }  \,  \left (  \frac{|\bigcup_iR_i|}{|R|}\right)^{\frac{\delta p^* }{ nM'} } |R|^\frac{\delta p^* }{ nM'}. \\
  \end{eqnarray*}
Therefore,
\begin{eqnarray*}
\sum_{i}a(R_i)^{p^*}w(R_i) &\le & 
\left( \frac{[w]_{A_{q,\cR}}|R|^{q}}{w(R) }   \right)  ^{ \frac{\delta p^*}{nqM}   }    
 \mu(R)^{  \frac{p^*}{p} }  \, \frac{|R|^{\frac{\delta p^* }{ nM'} }}{e(R)^{\delta p^*}}\, 
 \left (  \frac{|\bigcup_iR_i|}{|R|}\right)^{\frac{\delta p^* }{ nM'} }  \\
&= &  [w]_{A_{q,\cR}}^{\frac{\delta p^*}{nqM} }
\frac{|R|^{\frac{\delta p^*}{n}}}{w(R)^{  \frac{\delta p^*}{nqM}} }
\mu(R)^{ \frac{p^*}{p} }   \frac{1}{e(R)^{\delta p^*}}\, \, \left (  \frac{|\bigcup_iR_i|}{|R|}\right)^{\frac{\delta p^* }{ nM'} } \\
&= &  
[w]_{A_{q,\cR}}^{\frac{\delta p^*}{nqM} }
a(R)^{p^*}w(R)\, \left (  \frac{|\bigcup_iR_i|}{|R|}\right)^{\frac{\delta p^* }{ nM'} } .\\
\end{eqnarray*}

\end{proof}

\subsection{Optimal polynomials}\label{sec:polynomials}
We also include here some preliminaries on optimal polynomials.
Given a rectangle  $R\in \cR$ and an integer $m\ge 0$, we consider the space $\mathcal{P}_m$ of polynomials of degree at most $m$ in $n$ variables endowed with the inner product given by 
$$
<f,g>_R:=\avgint_R fg dx.
$$
There is an orthonormal basis with respect to this inner product that we will denote by $\{\phi_\alpha\}$, being $\alpha=(\alpha_1,\dots,\alpha_n)$ a multiindex of non negative integers such that $|\alpha|=\alpha_1+\cdots+\alpha_n\le m$. An important feature is that 
\begin{equation}\label{eq:equiv-norm-Pm}
\|\phi_r\|_{L^\infty}\le C\left (\frac{1}{|R|}\int_R |\phi_r|^2 dx\right) ^{1/2}=C,
\end{equation}
since the space $\mathcal{P}_m$ is finite dimensional and therefore all norms are equivalent. Let $P_R$ the projection defined by the formula
\begin{equation}\label{eq:PR}
P_R(f)=\sum_r \left (\frac{1}{|R|}\int_R f\phi_r dx\right )\phi_r.
\end{equation}

We clearly have from \eqref{eq:equiv-norm-Pm} that 
\begin{equation}\label{eq:Linfty-PR}
\|P_Rf\|_{L^\infty}\le N C^2 \avgint_R |f|,
\end{equation}
where $N$ depends on $m$. Moreover, as it is the case when $m=0$ and the projection is over the constants, we have the following optimality property:
\begin{equation*}
\inf_{\pi\in\mathcal{P}_m}\left (\avgint_R |f-\pi|^p\right) ^{1/p} \approx \left (\avgint_R|f-P_Rf|^p \right )^{1/p}.
\end{equation*}

\subsection{\texorpdfstring{$A_{\infty,\cR}$}{Ainfty} weights}\label{sec:weights}

We recall that the ``strong" $A_{\infty,\cR}$ class of weights is usually defined as  
$$ A_{\infty,\cR}=\bigcup_{p>1}A_{p,\cR}. $$
There are several different characterizations of the belonging of a weight to this class. We choose here to work with the so called Fujii-Wilson constant defined as
\begin{equation}\label{eq:AinftycR}
   [w]_{A_{\infty,\cR}}:=\sup_{R\in \cR}\frac{1}{w(R)}\int_R M_s(w\chi_R )\ dx <\infty.
\end{equation}
Although implicit in \cite{Fujii},  this definition was initially considered in \cite{HP} and used to derive a quite sharp Reverse H\"older Inequality  (RHI) estimate. 
We cite here an improved version from \cite{HPR1}, which is similar to the one obtained \cite{LPR} in the context of rectangles. The proof is an easy adaptation from cubes to rectangles and therefore will be omitted.

\begin{theorem}\label{thm:Ainfty-RHI-cR}
Let $w\in A_{\infty,\cR}$ in $\mathbb{R}^n$. Then for any rectangle $R\in \cR$,
\begin{equation*}
\avgint_{R} w^{1+\varepsilon}\ dx  \le 2\left(\avgint_{R} w \ dx\right)^{1+\varepsilon},
\end{equation*}
for any $\varepsilon>0$  such that $0<\varepsilon \leq \frac{1}{2^{n+1}[w]_{A_{\infty,\cR}}-1 }$. 
\end{theorem}

\section{Proofs Part I: Analysis in  \texorpdfstring{$\cR$}{cR}}\label{sec:proofscR}

In this section we will present the proof of our main contributions within the family of rectangles $\cR$ ($n$-fold product of one dimensional intervals). 

We begin by proving Theorem \ref{thm:AutomejoraweakcR}.  To do this we need some previous results.

Given a rectangle $R\in \cR$, $M^d_R$ will stand for the dyadic maximal function adapted to the rectangle $R$. That is, 
$$
M^d_R(f)(x):=\sup_{J\ni x}\avgint_J |f(y)|\ dy \qquad J\in \mathcal{D}(R)
$$
where $\mathcal{D}(R)$ is the family of dyadic subrectangles of $R$. 

We will also use the appropriate sharp maximal function defined by
\begin{equation*}
M_{m}^\sharp f(x) = \sup_{R\ni x} \avgint_R |f(y)-P_Rf|\,dy,
\end{equation*}
where $m$ is the degree and $P_R f$ is the polynomial approximation of degree less than $m$.

We will need the following result which is similar to the one obtained in \cite{CantoPerez} where several extensions of the classical John-Nirenberg's theorem has been derived. Our result is the context of the basis $\cR$ but the proof is essentially the same as the one given in \cite{CantoPerez} for cubes. It is also an extension to $\cR$ of one of the main results of Karagulyan in \cite{Karagulyan}. 
  The key point is to use a Calder\'on--Zigmund (C--Z) dyadic decomposition adapted to a given fixed rectangle $R\in\cR$. The validity of such a C--Z decomposition is immediate from the cubic case: the dyadic structure is exactly the same in both cases. 
In addition, we will introduce the following notation for $w \in A_{\infty, \cR}$, $r>1$ and for any rectangle $R \in \cR$:
\[ 
w_r(R)= |R| \left( \avgint_R w^r\right)^\frac{1}{r} = |R|^{\frac{1}{r'}} \left( \int_R w^r\right)^\frac 1r.
\]

\begin{theorem}\label{thm:TeoremacR}
Let $f\in L_{loc}^{1}(\mathbb{R}^{n})$ and let $1\leq p <\infty$ and $1<r<\infty$. Then there is a dimensional constant $c$ such that for any  rectangle $R\in \cR$ the following estimate holds
\begin{equation}\label{eq:John-NirenbergcR-general}
\left(  \frac{1}{w_r(R)} \int_R \left(\frac{M^d_{R}(f-P_Rf)(x)}{M_{m}^\sharp f(x)}\, \right)^p w(x)dx\right)^\frac{1}{p} \leq c \, pr'.
\end{equation}
Hence, if further $w\in A_{\infty, \cR}$ we have 
\begin{equation}\label{eq:John-NirenbergcR-Ainfty}
\left(  \frac{1}{w(R)} \int_R \left(\frac{M^d_{R}(f-P_Rf)(x)}{M_{m}^\sharp f(x)}\, \right)^p w(x)dx \right)^\frac{1}{p} 
\leq c\, p\,[w]_{A_{\infty,\cR}}.
%
%\int_R \left( \frac{M_{\cR}(f-f_R)(x)}{M^\sharp f(x)}\, \right)^p w(x)dx
%
\end{equation}
\end{theorem}

The proof of (\ref{eq:John-NirenbergcR-general}) follows the same ideas as in the proof of \cite[Theorem 7.1]{CantoPerez}. If we assume the extra condition of $w$ being an $A_\infty$ weight, then we can apply the reverse Hölder inequality for $\cR$ stated in Theorem \ref{thm:Ainfty-RHI-cR} to get \eqref{eq:John-NirenbergcR-Ainfty}.

\begin{corollary}\label{COROLARIOcR}
Let $f$ and   $w$ as before. Then for any  rectangle $R\in \cR$
\[ 
\norm{\frac{\displaystyle M^d_{R}(f-P_Rf)}{\displaystyle M_{m}^\sharp f}}_{\exp L(R, \frac{wdx}{w(R)})} \leq c\, [w]_{A_{\infty,\cR}}
\]
and hence there exist dimensional constants $c_1,c_2>0$ such that for any  rectangle $R\in \cR$ and $\lambda, \gamma >0$ 
\begin{equation*}   % \label{good-lambda}
w\big( \{x\in R: M^d_{R}(f-P_Rf)> \lambda, M_{m}^\sharp f(x) \leq \gamma \lambda \} \big)  \leq c_1\,
{e^{\frac {-c_2}{ \gamma  [w]_{A_{\infty,\cR}} }}   }\, w(R).
\end{equation*} 
\end{corollary}

\begin{proof}[Proof of Theorem \ref{thm:AutomejoraweakcR}]

With all these ingredients, the proof is an adaptation of the main result in \cite{FPW98} combined wth the exponential decay 
good-$\lambda$ type formula from Corollary  \ref{COROLARIOcR} very much as in \cite{CantoPerez}. We omit the details. 

\end{proof}

We will provide here the proof of our main general result Theorem \ref{thm:AutomejorastrongcR}. Before to proceed with the proof we recall that what we need to prove is that the quantity
$$
X=\sup_{R\in \cR} \left (\frac{1}{w(R)}\int_{R}\left |\frac{f-P_{R}f}{a(R)}\right |^p \, w dx \right )^{1/p}
$$
is bounded with a precise control on the bound. A brief comment on a technical step in the proof is necessary here. Since the argument involves the trick of bounding $X$ by some variation of itself, the partial step of proving that it is finite is crucial. In many standard situations, this can be achieved by assuming that the function $f$ is bounded and also by perturbing the functional $a$ to avoid any possible degeneracy considering $a_\varepsilon(R):=a(R)+\varepsilon$, $\varepsilon>0$. Once a uniform bound involving $a_\varepsilon$ is obtained, a limiting argument proves that the original $X$ is also finite. The problem is that the hypothesis on $a$ satisfying a $SD_{p,\cR}^s$ condition is not shared by the perturbed new functional $a_\varepsilon$. However, some weaker variant is indeed true and it is sufficient for the proof. Here (and only here!) is where the $A_\infty$ condition pops in, remarkably absent in the quantitative estimate. We will omit the technical details, we refer the reader to \cite{PR-Poincare} for a detailed discussion.

\begin{proof}[Proof of Theorem \ref{thm:AutomejorastrongcR}]
For the proof we follow the ideas from \cite{PR-Poincare}, starting with a standard local C--Z decomposition 
in the context of $\cR$.

We consider the local C-Z decomposition  of $\frac{|f-P_{R}f|}{a(R)}$ relative to $R$ at level $L$ for a large universal constant $L>1$ to be chosen. Let $\mathcal{D}(R)$ be the family of dyadic subrectangles of $R$ which also belong to $\cR$. The C--Z decomposition yields a collection $\{R_{j}\}_j$ of rectangles such that $R_j\in \mathcal{D}(R)$, maximal with respect to inclusion, satisfying
\begin{equation}\label{eq:CZ1}
L  < \frac{1}{|R_{j}|  }\int_{R_{j}} \frac{|f-   P_{R}f | }{{a(R)}} \, dy. 
\end{equation}
Then, if $R$ is dyadic with $R \supset R_j$ 
\begin{equation}\label{eq:CZ2}
\frac{ 1 }{|R|  }
\int_{R} \frac{|f-P_{R}f|}{{a(R)}} \, dy   \leq L
\end{equation}
and hence 
\begin{equation}\label{eq:CZ3}
L  < \frac{ 1 }{|R_{j}|  }
\int_{R_{j}} \frac{ |f-P_{R}f| }{{a(R)}} \, dy \leq L\,2^{n}
\end{equation}
for each integer $j$.  Also note that 
$$  \left \{x\in R: M_R^d\left ( \frac{|f-P_{R}f|}{{a(R)}}\chi_{  R }\right )(x) > L   \right \} = \bigcup_{j}R_j=:\Omega_L
$$
where $M^d_R$ stands for the dyadic maximal function adapted to the rectangle $R$. That is, 
$$
M^d_R(f)(x):=\sup_{J\ni x}\avgint_J |f(y)|\ dy\qquad x\in J, J\in \mathcal{D}(R).
$$
Then, by the Lebesgue differentiation theorem it follows that
\begin{equation}\label{eq:notinOmega}
\frac{|f(x)-P_{R}f(x)|}{{a(R)}} \leq L   \qquad              a.e. \ x \notin \Omega_L.
\end{equation}

Also, observe that by \eqref{eq:CZ1} (or the weak type $(1,1)$ property of $M^d_R$) and recalling our starting assumption \eqref{eq:UnWeightedStartingPointL1}, we have that $\{R_j\}_j\in S(L)$, namely
\begin{equation*}\label{eq:CZ4a}
|\Omega_L|=|\bigcup_{j}R_j | < \frac{|R|}{L}.
\end{equation*}

Now, given the C-Z decomposition of the cube $R$, we decompose the function $\dfrac{f-P_Rf}{{a(R)}}$ as
\begin{eqnarray*}\label{eq:CZ-f-fR}
f-P_{R}f & = & \hspace{-.1cm}\left (f-P_{R}f\right )\chi_{\Omega_L^c}
+\sum_j \left (P_{R_j}f-P_{R}f\right )\chi_{R_j}
+\sum_j \left (f-P_{R_j}f\right )\chi_{R_j}\\
& = & A_1 + A_2 + A_3.
\end{eqnarray*}

For $A_1$ and $A_2$ we have a pointwise estimate. The bound for $A_1$ follows from \eqref{eq:notinOmega}.
Also note that $P_{R_j}P_R=P_{R}$ and recall the bound \eqref{eq:Linfty-PR}. Then we can control $A_2$ as
\begin{eqnarray*}
|A_2(x)| & \le & \sum_j \left | (P_{R_j}f(x)-P_{R}f(x))\right |\chi_{R_j}(x)\\
& = & \sum_j \left | P_{R_j}(f-P_{R}f)\right (x)|\chi_{R_j}(x)\\
& \le & \sum_j NC^2 \avgint_{R_j}|f-P_{R}f|\chi_{R_j}(x)\\
& \le & NC^2 2^n L a(R)
\end{eqnarray*}
since for any given $x \in\Omega$ there is only one $j$ such that $x\in R_j$. So far we have pointwise bounds for $A_1$ and $A_2$. Now we move on to the desired $L^p$ norm estimate. Consider the local weighted $L^p$ space on $R$ with the measure $\mu:=\frac{w(x)dx}{w(R)}$. By the triangle inequality, we have that
\begin{eqnarray*}
\left \|\frac{f-P_{R}f}{a(R)}\right \|_{L^p_\mu}  & \le  &\left \|\frac{A_1}{a(R)}\right \|_{L^p_\mu} 
+\left \|\frac{A_2}{a(R)}\right \|_{L^p_\mu} 
+\left \|\sum_j \frac{f-P_{R_j}f}{a(R)}\chi_{R_j}\right \|_{L^p_\mu} \\
& \le  & \tilde{c}L + \left (\frac{1}{w(R)}\int_R\sum_j \frac{|f-P_{R_j}f|^p}{a(R)^p}\chi_{R_j}wdx\right )^{\frac{1}{p}}.
\end{eqnarray*}	
Since the $R_j$'s are disjoint, we can compute
\begin{eqnarray*}
\int_R\sum_j \frac{|f-P_{R_j}f|^p}{a(R)^p}\chi_{R_j}wdx  & = & \sum_j\frac{a(R_j)^pw(R_j)}{a(R)^pw(R_j)}\int_{R_j}\frac{|f-P_{R_j}f|^p}{a(R_j)^p}wdx\\
& \le & X^p \sum_j\frac{a(R_j)^pw(R_j)}{a(R)^p},
\end{eqnarray*}	
where $X$ is the quantity defined as
\begin{equation*}
X=\sup_{R\in \cR}\left  (\frac{1}{w(R)}\int_{R}\frac{|f-P_{R}f|^p}{a(R)^p}wdx\right )^{1/p}.
\end{equation*}
Averaging over $w(R)$ and raising to the $1/p$ power, we obtain
\begin{eqnarray*}
\left (\frac{1}{w(R)}\int_{R}\left |\frac{f-P_{R}f}{a(R)}\right |^p \, w dx \right )^{1/p} & \le &  \tilde{c}L+X\left (\sum_j\frac{a(R_j)^pw(R_j)}{a(R)^pw(R)}\right )^\frac{1}{p}\\
& \le & \tilde{c}L + X \frac{\|a\|}{L^{1/s}}.
\end{eqnarray*}	
Taking the supremum over all $R\in \cR$, we finally obtain
\begin{equation}\label{eq:isolating-X}
X \le \tilde{c}L + X \frac{\|a\|}{L^{1/s}},
\end{equation}
for any $L>1$.
Choosing here an appropriate $L$ as $L=e\max\{\|a\|^s,1\}$ and using the elementary fact that $\left (e^{1/s}\right )'\le 1+s$, we finally get that, for any $R\in \cR$,
\begin{equation*}
\left (\frac{1}{w(R)}\int_{R}\left |f-P_{R}f\right |^p \, w dx \right )^{1/p} \le c_n (1+s)\max\{\|a\|^s,1\}\,a(R),
\end{equation*}
which is the desired estimate. 
\end{proof}

Now we are ready to prove the consequences that can be derived from our general theorem. We start by proving Theorem \ref{thm:PoSo-Aq-diam-delta}, where we exploit the membership of the weight $w$ to a better $A_{q,\cR}$ class to reach a higher Poincar\'e-Sobolev exponent.

\begin{proof}[Proof of Theorem \ref{thm:PoSo-Aq-diam-delta}]
Recall that the functional from \eqref{eq:model-d(Q)^delta-a(Q)} satisfies a $SD_{p,\cR}^{n/\delta}$ condition with norm 1 (see Example \ref{ex:mainexample}). This is sufficient to prove $(p,p)$-Poincar\'e inequalities by using Theorem \ref{thm:AutomejorastrongcR}. Here, in order to achieve a higher integrability, 
the problem reduces to prove a better smallness preservation condition that comes with a price: the norm with respect to the better parameters is no longer uniformly bounded on the weights. Hence, a careful optimization has to be done. 
We will use the main estimate from Lemma \ref{lem:p^*},  choosing $M>1$ with $s=\frac{nM'}{\delta}$, namely 
\begin{equation}\label{ineq:||a||}
\|a\|:=\|a\|_{SD_{p^*_M,\cR}^{ \frac{nM'}{\delta} }(w)} \leq [w]_{A_{q,\cR}}^{ \frac{ \delta}{nqM} }
\end{equation}
where \,$ \displaystyle \frac{1}{p} -\frac{1}{ p_M^*}=\frac{\delta}{nqM}.$ The choice will depend on $[w]_{A_{q,\cR}}$ so that the expression $s\|a\|^s$ remains bounded. We choose 
\begin{equation}\label{eq:M}
M= 1+ \log[w]_{A_{q,\cR}}^\frac{1}{q}.
\end{equation}
Observe that both $M$ and $M'=	\frac{1+ \log[w]_{A_{q,\cR}}^\frac{1}{q}}{ \log[w]_{A_{q,\cR}}^\frac{1}{q}}> 1 $ and also that this choice of $M$ implies that the exponent $p^*_M$ is exactly the value $p^*_w$ from \eqref{eq:ptimes-Aq}.
Hence applying Theorem \ref{thm:AutomejorastrongcR}  with $p$ replaced by $p^*_w$  and using estimate \eqref{ineq:||a||} we obtain 
\begin{eqnarray*}
\left( \frac{1}{ w(R)  } \int_{ R }   |f -P_{R}f|^{p_w^*}    \,wdx\right)^{\frac{1}{p_w^*}}  & \leq  & C_n s\, \|a\|^{s}\, a(R) \\
& \le & C_n \frac{M'}{\delta} [w]_{A_{q,\cR}}^{\frac{\delta}{nqM}\frac{nM'}{\delta }} \, a(R)\\
& \le & C_n\frac{M'}{\delta}  [w]_{A_{q,\cR}}^{\frac{1}{ \log [w]_{A_{q,\cR}}  }} \, a(R)\\
& = & c_n \frac{1+ \log[w]_{A_{q,\cR}}^\frac{1}{q}}{ \log[w]_{A_{q,\cR}}^\frac{1}{q}}\frac{1}{\delta}\, a(R).
\end{eqnarray*}
This yields the proof of Theorem \ref{thm:PoSo-Aq-diam-delta} providing a good estimate for weights with large nontrivial $A_{q,\cR}$ constants. A simple computation shows that for any weight such that $[w]_{A_{q,\cR}}>1$, if we choose $K$ large such that $ [w]_{A_{q,\cR}}=e^{\frac{q}{K}}$, then we obtain a cleaner inequality
\begin{equation*}
\left( \frac{1}{ w(R)  } \int_{ R }   |f -f_{R}|^{p_w^*}    \,wdx\right)^{\frac{1}{p_w^*}} 
\le 
 c_n (K+1)\frac{1}{\delta}\, a(R). 
\end{equation*}
\end{proof}

Now we present the proof of Theorem \ref{thm:PoSo-Aq-diam-delta-weak} covering, in particular, the case of flat weights.

\begin{proof}[Proof of Theorem \ref{thm:PoSo-Aq-diam-delta-weak}]

We can assume that the weight $w$ satisfies $[w]_{A_q,\cR}\leq e^q$, since in the other case we simply use the previous Theorem \ref{thm:PoSo-Aq-diam-delta}.
We claim that for any family $\{R_i\}_i$ of pairwise disjoint dyadic subrectangles of $R$, the following inequality holds:
\begin{equation}\label{eq:p1*-importantEstimate}
\left(\sum_{i}a(R_i)^{ p^*_{1} }\,\frac{w(R_i)}{w(R)} \right)^{\frac{1}{p^*_{1}} } \leq  e^{ \frac{\delta}{n}  }
\, a(R), 
\end{equation}
recalling the notation $ \displaystyle \frac{1}{p} -\frac{1}{ p_1^*}=\frac{\delta}{n}\frac{1}{q}$. This means that $a\in D_{p^*_{1},\cR}(w)$ 
with $\|a\|_{D_{p^*_{1},\cR}(w)} \leq   e^{ \frac{\delta}{n} }$ according to Definition \ref{def:Dp(w)}.  Let's postpone the proof of the claim for the time being.  Now, applying Theorem \ref{thm:AutomejoraweakcR}, we get
\begin{eqnarray*}
\| f-f_R\|_{L^{p^*_{1},\infty}\big( R, \frac{w\,dx}{w(R)}\big)} &\leq&  c\, p^*_{1}\, [w]_{A_{\infty,\cR}}\,\|a\|_{D_{p^*_{1},\cR}(w)} \, a(R)\\
&\leq & c\, p^*_{1}\, [w]_{A_{q,\cR}}\,\|a\|_{D_{p^*_{1},\cR}(w)} \, a(R)\\
&\leq & c\, p^*_{1}\,e^q\,  e^{ \frac{\delta}{n} } \, a(R).  
\end{eqnarray*}
Consider here the same choice of $M=1+\log [w]^{\frac{1}{q}}_{A_{q,\cR}}$ as in the proof of Theorem \ref{thm:PoSo-Aq-diam-delta} above. Since  $p^*_{1}>p^*_M$ (just note that $p^*_M$ is a decreasing function on $M$), by Jensen's inequality which holds 
at weak level (simply use that the inner part of what is inside $()^{\frac1p}$ is less or equal than one) we have,

\begin{eqnarray*}
\| f-f_R\|_{L^{p^*_M,\infty}\big( R, \frac{w\,dx}{w(R)}\big)} & \leq  & \| f-f_R\|_{L^{p^*_{1},\infty}\big( R, \frac{w\,dx}{w(R)}\big)}\\
&\le &
c\, p^*_{1}\,e^q\, e^{ \frac{\delta}{n}  } \, a(R).  
\end{eqnarray*}
This is the desired weak type estimate.  To conclude the proof of the theorem we need to check claim \eqref{eq:p1*-importantEstimate}. But this is a variant of the proof of Lemma \ref{lem:p^*}, indeed by the key eccentricity property in Lemma \ref{lem:eccentr},  
\begin{eqnarray*}
\sum_{i}a(R_i)^{p^*_{1}}w(R_i) & = & \sum_{i} \mu(R_i)^{ \frac{p^*_{1}}{p} }    \left (    \frac{d(R_i)^\delta}{w(R_i)^{\frac{1}{p} -\frac{1}{p^*_{1}}} } \right )^{ p^*_{1}  }          \\
&= & \frac{1}{e(R)^{\delta p^*_{1}}}\sum_{i} \mu(R_i)^{\frac{p^*_{1}}{p}}    \left(    \frac{ |R_i|^\frac{\delta}{n} }{w(R_i)^{ \frac{\delta}{qn}  } } \right )^{ p^*_{1} } \\
&= & \frac{1}{e(R)^{\delta p^*_{1}}}\sum_{i} \mu(R_i)^{\frac{p^*_{1}}{p}}    \left(    \frac{ |R_i| }{w(R_i)^{ \frac{1}{q}  } } \right )^{ \frac{\delta p^*_{1}}{n} } \\
&\leq &\frac{1}{e(R)^{\delta p^*_{1}}}\sum_{i} \mu(R_i)^{\frac{p^*_{1}}{p}}    \left( [w]_{A_{q,\cR}}^{ \frac{1}{q}  }  \frac{ |R| }{w(R)^{\frac{1}{q}} } \right )^{ \frac{\delta p^*_{1}}{n} } \\
& \leq &  \frac{   e^{\frac{\delta p^*_{1}}{n} } }{e(R)^{\delta p^*_{1}}}
\left( \frac{|R|^{q}}{w(R) }   \right)  ^{ \frac{\delta p^*_{1}}{nq}   }  \  
  \sum_{i} \mu(R_i)^{\frac{p^*_{1}}{p}}\\
& \leq &   \frac{   e^{\frac{\delta p^*_{1}}{n} } }{e(R)^{\delta p^*_{1}}}
\left( \frac{|R|^{q}}{w(R) }   \right)  ^{ \frac{\delta p^*_{1}}{nq}   }  \  
\mu(R)^{\frac{p^*_{1}}{p}}\\
& =&     e^{\frac{\delta p^*_{1}}{n} }
a(R)^{p_{1}^*}w(R)
\\
\end{eqnarray*}
since we assumed that $[w]_{A_q,\cR}\leq e^q$.  This gives the claim \eqref{eq:p1*-importantEstimate} finishing the proof of the theorem.
\end{proof}

\begin{proof}[Proof of Corollary \ref{cor:ptimes-Aq-m-derivatives}]
In order to apply our self improving method from Theorem \ref{thm:PoSo-Aq-diam-delta} or Theorem \ref{thm:PoSo-Aq-diam-delta-weak}, we need a starting point as \eqref{eq:Sarting-PoSo-Aq} where the functional $a$ is of the form \eqref{eq:model-d(Q)^delta-a(Q)}. Regarding the starting point, we know from the work of Chua in   \cite[Lemma  2.5]{Chua}    that the following unweighted inequality holds \footnote{The results in \cite{Chua} are derived for cubes but the same results hold for rectangles as well.}
\begin{equation*}\label{eq:HigherOrderPoincare-Unweighted}
\frac{1}{ |R| }\int_{R} \absval{f-  P_Rf} \,dx \le C\,
\frac{ d(R)^{m} }{  |R| }
\int_{R} \absval{\nabla ^{m}f} \,dx,
\end{equation*}
where $P_Rf$  is the ``approximating" polynomial as defined in \eqref{eq:PR} and for a class of Sobolev functions with appropriate smoothness, namely $f\in W^{m,1}(R)$. This result follows from the well known first order $(1,1)$-Poincar\'e inequality on convex sets. From this inequality and using the standard   argument of introducing an $A_{p,\cR}$ weight we derive that
\begin{equation}\label{eq:HigherOrderPoincare-Weighted}
\avgint_{R} \absval{f -   P_Rf} \,dx \le C\,
[w]^{1/p}_{A_{p,\cR}}d(R)^{m} 
\left (\frac{1}{w(R)}\int_{R} \absval{\nabla ^{m}f}^p w\,dx\right )^{\frac1p}.
\end{equation}
Hence, we have our starting point 
\begin{equation*}
\frac{1}{ |R| }\int_{R} \absval{f-      P_Rf  }\,dx \leq a(R),
\end{equation*}
where the functional $a$ is defined by 
\begin{equation*}
a(R)=C\,[w]^{\frac1p}_{A_{p,\cR}}d(R)^{m} 
\left (\frac{1}{w(R)}\int_{R} \absval{\nabla ^{m}f}^p w\,dx\right )^{\frac1p}
\end{equation*}
which is a particular case of \eqref{eq:Model-a(R)}. An application of Theorem \ref{thm:PoSo-Aq-diam-delta} gives the result for nontrivial weights with $A_q$ constant bounded from below by $e^q$. For a result valid for flat weights, including the unweighted classical Poincar\'e-Sobolev inequality for the Lebesgue measure, we can only use Theorem \ref{thm:PoSo-Aq-diam-delta-weak} to get a weak inequality. 
\end{proof}

Now we move on to prove Corollary \ref{cor:fractional}. 

\begin{proof}[Proof of Corollary \ref{cor:fractional}]
Again, the key resides in proving a suitable starting point. 
Consider, for $\delta>0$, the functional defined by the expression 
\begin{equation*}
a(R)   := [w]^\frac{1}{p}_{A_{p,\cR}}\frac{d(R)^{\delta}}{e(R)^\frac{n}{p}}\left(\frac{1}{w(R)}\int_R A(R,x)\,w(x)dx \right)^{\frac1p}
\end{equation*}
 where
 \begin{equation}\label{eq:a(Q)-fractional}
 A(R,x)=\int_R \frac{|f(x)-f(y)|^p}{|x-y|^{\delta p+n}}\,w(y)dy.
 \end{equation}

Note that  by Hölder's inequality combined with the definition, again,  of the $A_{p,\cR}$ condition together with the concept of eccentricity we obtain
\begin{eqnarray*}
\avgint_R |f-f_R|&\approx& \frac{1}{|R|}\int_R \frac{1}{|R|}\int_R |f(x)-f(y)|\,dy\,dx \\
&\leq& [w]^\frac{1}{p}_{A_{p,\cR}} \left(\frac{1}{w(R)|R|}\int_R \int_R |f(x)-f(y)|^p \,dy\,w(x)dx\right)^\frac{1}{p}\\ 
&\leq&  [w]^\frac{1}{p}_{A_{p,\cR}}\,\left(  \frac{d(R)^{\delta p}}{e(R)^nw(R)}\int_R \int_R \frac{|f(x)-f(y)|^p}{|x-y|^{\delta p+n}} \,dy\,w(x)dx\right)^\frac{1}{p}\\
& = &   [w]^\frac{1}{p}_{A_{p,\cR}} \,\frac{d(R)^{\delta }}{e(R)^\frac{n}{p}}\left(  \frac{1}{w(R)}\int_R \int_R \frac{|f(x)-f(y)|^p}{|x-y|^{\delta p+n}} \,dy\,w(x)dx\right)^\frac{1}{p}\\
& = & a(R).
\end{eqnarray*}
Note that the functional $a$ here is not \emph{exactly} the same as in the model example \eqref{eq:model-d(Q)^delta-a(Q)} due to the presence of the eccentricity in front. But since we are testing the smallness preservation on families of dyadic rectangles on a fixed rectangle $R$, any function on the eccentricity moves out of the sum and therefore this kind of functional also satisfies Lemma \ref{lem:p^*}. Again, the result follows from Theorem \ref{thm:PoSo-Aq-diam-delta}.
\end{proof}

Next we present here the proof of Theorem \ref{thm:delta-StartingPoint}. Recall that the idea is that even only being able to start from a weaker starting point \eqref{eq:delta-StartingPoint}, in the end we arrive at the same place, namely to the usual $L^1$ oscillation initial point. 
 
 \begin{proof}[Proof of Theorem \ref{thm:delta-StartingPoint}]
Note that proving \eqref{eq:L^1-goal} is the same as proving that the following quantity
\begin{equation}\label{DefX}
X:=\sup_{R\in \cR}\inf_{c\in \mathbb R}\avgint_R \frac{|f(x)-c|}{a(R)} dx 
\end{equation}
is finite. There are two reductions that make this trivial: the function $f$ being bounded and the functional $a$ being uniformly away from zero. Both can be formalized in a standard way. For the functional $a$ it suffices to consider the perturbed functional $a_\varepsilon:=a+\varepsilon$ for $\varepsilon>0$ since this does not affect the properties assumed on $a$. This observation may (and should) be compared to  the remark made prior to the proof of Theorem \ref{thm:AutomejorastrongcR}. The absence of weights makes these reductions much easier.

The argument below shows how to truncate the function $f$. Therefore, we can always assume that the supremum defining $X$ is finite.  Let us start then by observing that we may assume that the function $f$ is bounded. More precisely, suppose that \eqref{eq:delta-StartingPoint} holds and consider $m\in \mathbb{N}$. Then the truncation $f_m:=\min\{|f|; m\}$ also verifies \eqref{eq:delta-StartingPoint}. To see why, consider $\lambda \in \mathbb{R}$ and also truncate it by introducing $\lambda_m:=\min\{|\lambda|;m\}$. We can compute directly both minima  and obtain:
 \begin{eqnarray*}
 f_m-\lambda_m & = & \min\{|f|; m\} - \min\{|\lambda|; m\}\\
& = & \frac{1}{2}\left (|f|+ m -||f|- m| - \left (|\lambda|+m-||\lambda|-m|\right) \right )\\
& = & \frac{1}{2}\left (|f| -|\lambda| - ||\lambda|-m|-||f|-m| \right ).
 \end{eqnarray*}
 Hence, taking absolute values, 
  \begin{eqnarray*}
| f_m-\lambda_m| & \le &  \frac{1}{2}\left (||f| -|\lambda|| + |||\lambda|-m|-||f|-m| |\right )\\
& \le & \frac{1}{2}\left (|f -\lambda| + ||\lambda|-|f| |\right )\\
& \le & \frac{1}{2}\left (|f -\lambda| + |f-\lambda|\right )\\
& \le & |f -\lambda| 
 \end{eqnarray*}
 for all $m\in \mathbb{N}$ and for all $\lambda>0$. This pointwise estimate yields the following
 \begin{equation*}
 \inf_{c\in \mathbb R}\left (\avgint_R |f_m(x)-c|^\delta\right )^{\frac{1}{\delta}}\le \left (\avgint_R |f_m(x)-\lambda_m|^\delta\right )^{\frac{1}{\delta}}\le \left (\avgint_R |f(x)-\lambda|^\delta\right )^{\frac{1}{\delta}},
 \end{equation*}
 also for any $\lambda>0$. Therefore we obtain that
 \begin{equation*}
  \inf_{c\in \mathbb R}\left (\avgint_R |f_m(x)-c|^\delta\right )^{\frac{1}{\delta}}\le \inf_{\lambda\in \mathbb R}\left (\avgint_R |f(x)-\lambda|^\delta\right )^{\frac{1}{\delta}}\le a(R).
 \end{equation*}
 This means that the same initial hypothesis also holds for any truncation of $f$ at height $m$, independently of $m$.

Now we start with the proof of inequality \eqref{eq:L1-starting} assuming that the function $f$ is bounded. Given a rectangle $R\in \cR$, we have by hypothesis that, for some $c_R\in \mathbb R$, 
\begin{equation*}
\left (\avgint_R \left |\frac{f(x)-c_R}{a(R)}\right |^\delta dx\right )^{\frac{1}{\delta}}\le 2.
\end{equation*}
We perform the standard dyadic C-Z decomposition of the function $\left |\frac{f-c_R}{a(R)}\right |^\delta $  adapted to $R$ at level $L>2$ to be chosen later. 
This means that we will have a collection $\{R_i\}_i$ of disjoint maximal dyadic rectangles satisfying 
\begin{equation}\label{eq:CZ-rectangles}
L< \left (\avgint_{R_i} \left |\frac{f(x)-c_R}{a(R)}\right |^\delta dx\right )^{\frac{1}{\delta}}\le 2^{\frac{n}{\delta}}L.
\end{equation}

Lets denote by $\Omega_L$ the union of such rectangles. As usual, this set is the level set of the dyadic maximal function relative to $R$:
\begin{equation}\label{eq:OmegaL-level-set}
\Omega_L=\left \{x\in R: M^d_R\left (\left |\frac{f-c_R}{a(R)}\right |^\delta \right )(x)>L^\delta\right \},
\end{equation}
where 
\begin{equation}\label{eq:MaximalDyadic}
M^d_R(g)(x)=\sup_{x\in P\in \mathcal{D}(R)}\avgint_P |g(y)|dy.
\end{equation}

Now, consider the following pointwise decomposition.
\begin{eqnarray*}
\left |\frac{f(x)-c_R}{a(R)}\right |^\delta & = & \left |\frac{f(x)-c_R}{a(R)}\right |^\delta\chi_{\Omega}(x) + \left |\frac{f(x)-c_R}{a(R)}\right |^\delta\chi_{\Omega^c}(x) \\
& \le &  \sum_i \left |\frac{f(x)-c_R}{a(R)}\right |^\delta\chi_{R_i}(x) + L^\delta
\end{eqnarray*}
since the Lebesgue differentiation theorem implies that outside $\Omega$ the function $\frac{f-c_R}{a(R)}$ is pointwise bounded by $L$. The usual argument at this stage would be to find a way to intercalate some average-like quantity related to $R_i$ instead of the original object $R$. We proceed as follows in the sum above. Denote $g_R=\frac{f-c_R}{a(R)}$, then

\begin{eqnarray*}
\sum_i |g_R(x)|^\delta\chi_{R_i}(x) &  =  & \sum_i \left ( |g_R(x)|^\delta-\avgint_{R_i}|g_R(y)|^\delta dy\right )\chi_{R_i}(x) \\
& & + \sum_i \avgint_{R_i}|g_R(y)|^\delta\chi_{R_i}(x) \\
& \le &  \sum_i \left ( \avgint_{R_i}\left |\frac{f(x)-f(y)}{a(R)}\right |^\delta dy\right )\chi_{R_i}(x) + 2^nL^\delta,
\end{eqnarray*}
by using the elementary bound $||a|^\delta-|b|^\delta|\le |a-b|^\delta$ for any $a,b\in\mathbb{R}$ and $0<\delta<1$. The second term is obtained using the maximality from \eqref{eq:CZ-rectangles}. Now, by convexity the $\delta$ power can be moved out from the integral and using that the rectangles are disjoint we obtain
\begin{equation*}
\sum_i |g_R(x)|^\delta\chi_{R_i}(x) \le  \left (\sum_i  \avgint_{R_i}\left |\frac{f(x)-f(y)}{a(R)}\right |dy\chi_{R_i}(x)\right )^\delta  + 2^nL^\delta.
\end{equation*}
Therefore, collecting all estimates, we get
\begin{equation*}
\left |\frac{f(x)-c_R}{a(R)}\right |^\delta\le \left (\sum_i  \avgint_{R_i}\left |\frac{f(x)-f(y)}{a(R)}\right |dy\chi_{R_i}(x)\right )^\delta  + 2^nL^\delta + L^\delta.
\end{equation*}

Now we can compute the desired $L^1$ norm. 
\begin{eqnarray*}
 \avgint_R\left |\frac{f(x)-c_R}{a(R)}\right |^{\delta\frac{1}{\delta}}& \le  &
\frac{2^{\frac{1}{\delta}-1}}{|R|}\sum_i  |R_i|\avgint_{R_i} \avgint_{R_i}\left |\frac{f(x)-f(y)}{a(R)}\right |dydx+ 2^\frac{n+2-\delta}{\delta}L.\\
\end{eqnarray*}

Now,  since the double average over $R_i$ of $f$ is less or equal than twice the infimum of the oscillations on $R_i$ we obtain, writing $C_\delta=2^\frac{n+2-\delta}{\delta}$
\begin{eqnarray*}
 \avgint_R\left |\frac{f(x)-c_R}{a(R)}\right |dx 
& \le & \frac{2^{\frac{1}{\delta}}}{a(R)|R|} \sum_i  a(R_i)|R_i|\inf_{c\in\mathbb{R}}\avgint_{R_i}\left |\frac{f(x)-c}{a(R_i)}\right |dy+C_\delta L \\
& \le& X \frac{2^{\frac{1}{\delta}}}{a(R)|R|} \sum_i  a(R_i)|R_i|+C_\delta L \\
\end{eqnarray*}
  where  $X$ is defined by \eqref{DefX}. Here there are two options, as it is presented in the statement of the theorem. In the case of the functional $a$ satisfying  the $SD_{1,\cR}^s$ condition for some $s>1$, we use this to obtain
\begin{eqnarray*}
\inf_{c\in \mathbb R} \avgint_R\left |\frac{f(x)-c}{a(R)}\right |dx  & \le &  \avgint_R\left |\frac{f(x)-c_R}{a(R)}\right |dx \\
& \le & 2^{\frac{1}{\delta}} X \frac{\|a\|_{SD_{1,\cR}^s } }{L^{1/s}}+2^\frac{n+2-\delta}{\delta}L.
\end{eqnarray*}
Therefore, taking the supremum over all rectangles $R$ on the left hand side we obtain
\begin{equation*}
X \le 2^{\frac{1}{\delta}} X \frac{\|a\|_{SD_{1,\cR}^s}}{L^{1/s}}+2^\frac{n+2-\delta}{\delta}L.
\end{equation*}

Now we proceed essentially as in \eqref{eq:isolating-X}. 
Choose $L$ large enough, namely $L=e\max\{\max\{2^{\frac{s}{\delta}}\|a\|_{SD_{1,\cR}^s}^s,1\}$ to note that  
$$
X\leq \frac{1}{e^{\frac1s}}X +2^\frac{n+2-\delta}{\delta}e \max\{2^{\frac{s}{\delta}}\|a\|_{SD_{1,\cR}^s}^s,1\}.
$$ 
As in \eqref{eq:isolating-X}, we use that  $(e^{\frac1s})'\le 1+s$ to conclude that 
\begin{equation*}
X\lesssim_{\delta,s}  \max\{ \|a\|^s_{SD_{1,\cR}^s},1 \} ,
\end{equation*}
which is the desired estimate.

The other possibility is that the functional $a$ satisfies a standard $D_p$ condition for $1<p<\infty$. In that case, using H\"older's inequality in the sum involving the $R_i$'s, we obtain \footnote{ Actually, we are using that $D_p \subset SD_{1,\cR}^s$  with  $  \|a\|_{SD_{1,\cR}^s } \leq   \frac{\|a\|_{ D_{p,\cR}   }   }{L^{ \frac{1}{p'}    }}$}
\begin{eqnarray*}
X & \le &  2^{\frac{1}{\delta}}X \frac{1}{a(R)|R|} \sum_i  a(R_i)|R_i|^{\frac{1}{p}+\frac{1}{p'}}+2^\frac{n+2-\delta}{\delta}L\\
& \le & 2^{\frac{1}{\delta}}X \left (\frac{\sum_i  a(R_i)^p|R_i|}{a(R)^p|R|} \right )^{1/p}\left (\frac{\sum_i  |R_i|}{|R|}\right )^{\frac{1}{p'}}+2^\frac{n+2-\delta}{\delta}L\\
& \le & 2^{\frac{1}{\delta}}X  \frac{\|a\|_{ D_{p,\cR}   }   }{L^{ \frac{1}{p'}    }} +2^\frac{n+2-\delta}{\delta}L.
\end{eqnarray*}
Once again, we are in the same situation as before and we can argue in the same way: choosing carefully the size of $L$, we obtain again an inequality of the form
\begin{equation*}
  X\lesssim_{\delta, p} \|a\|_{D_{p,\cR} }^{p'}.
\end{equation*}
 \end{proof}

\section{Analysis in \texorpdfstring{$\ccR$}{ccr}: Bi-paramater (1,1)-\texorpdfstring{Poincar\'e}{Poincare}} \label{sec:Prelim-ccR}

We start recalling the following $(1,1)$-Poincaré inequality 
\begin{equation}\label{eq:suma}
\avgint_{R} |f - f_{R}|\leq \ell(I_1)\avgint_{R} |\nabla_1 f| + \ell(I_2)\avgint_{R} |\nabla_2 f| \qquad R\in \ccR
\end{equation}
which is the statement of Lemma \ref{lem:(1,1)PI-ccR}. As already mentioned the idea  of considering this type of Poincaré inequality follows from the work of Shi and Torchinsky \cite{ShiTor}. 

We will prove inequality \eqref{eq:suma} in two ways, the first one using a point-wise estimate obtained by Lu and Wheeden in \cite{LuWheeden} and the second proof will follow as a consequence of a ``fractional" Poincaré type inequalities that we will show in Proposition \ref{pro:1-1FPIproductSpaces} and improve in Theorem \ref{thm:BBMbiparametrico}.

Recall that for a function $f$ defined on $\mathbb{R}^n=\mathbb{R}^{n_1}\times \mathbb{R}^{n_2}$, $\nabla_1f$ will denote the partial gradient of $f$ containing the $x_1$-derivatives and 
similarly for $\nabla_2f$ the partial gradient of $f$ containing the $x_2$-derivatives. From now on, $I_1$ will always denote a cube in $\mathbb{R}^{n_1}$ while $I_2$ will be a cube in $\mathbb{R}^{n_2}$. Let $\ell_1:=\ell(I_1)$ and $\ell_2:=\ell(I_2)$.

\begin{proposition}\label{pro:LW} \cite{LuWheeden}
Let $R\in \ccR$ be of the form $R=I_1\times I_2$.  If $f\in \text{Lip}(R),$ we have the following pointwise estimate
\begin{equation*}
|f(x_1,x_2) - f_{R}| \lesssim \avgint_{R} \frac{\ell_1|\nabla_1f(y_1,y_2)|+\ell_2|\nabla_2f(y_1,y_2)|}{\left (\frac{|x_1-y_1|^{2}}{\ell_1^2}+\frac{|x_2-y_2|^{2}}{\ell_2^2}\right )^{\frac{n-1}{2}}} \ dy_1 dy_2.
\end{equation*}
\end{proposition}

From this we provide the first proof of the $(1,1)$-Poincar\'e inequality \eqref{eq:suma}.

\begin{proof} [First proof of Lemma \ref{lem:(1,1)PI-ccR}]
Let's put $A:=\displaystyle\int_{R}|f(x_1,x_2)- f_{R}| \, dx_1 dx_2$. Then, using the pointwise estimate from Proposition \ref{pro:LW} we directly obtain
\begin{eqnarray*}
A
&\le &
\int_{R}\avgint_{R} \frac{\ell_1|\nabla_1f(y_1,y_2)|+\ell_2|\nabla_2f(y_1,y_2)|}{\left (\frac{|x_1-y_1|^{2}}{\ell_1^2}+\frac{|x_2-y_2|^{2}}{\ell_2^2}\right )^{\frac{n-1}{2}}} \ dy_1 dy_2 dx_1 dx_2\\
&\le & \int_{R}\frac{\ell_1|\nabla_1f(y_1,y_2)|}{|I_1||I_2|}\int_{R} \frac{1}{\left (\frac{|x_1-y_1|^{2}}{\ell_1^2}+\frac{|x_2-y_2|^{2}}{\ell_2^2}\right )^{\frac{n-1}{2}}} \ dx_1 dx_2 dy_1 dy_2\\
&&+ \int_{R}\frac{\ell_2|\nabla_2f(y_1,y_2)|}{|I_1||I_2|}\int_{R} \frac{1}{\left (\frac{|x_1-y_1|^{2}}{\ell_1^2}+\frac{|x_2-y_2|^{2}}{\ell_2^2}\right )^{\frac{n-1}{2}}} \ dx_1 dx_2 dy_1 dy_2\\
&\le & \ell_1\int_{R} |\nabla_1 f(y_1,y_2)| \ dy_1dy_2+ \ell_2\int_{R} |\nabla_2 f(y_1,y_2)| \ dy_1dy_2.
\end{eqnarray*}
In the above argument we have use the following inequality
\begin{eqnarray*}
\avgint_{R} \frac{ dx_1 dx_2}{\left (\frac{|x_1-y_1|^{2}}{\ell_1^2}+\frac{|x_2-y_2|^{2}}{\ell_2^2}\right )^{\frac{n-1}{2}}} 
&= & \avgint_{R} \frac{ dx_1 dx_2}{\left|(\frac{x_1}{\ell_1},\frac{x_2}{\ell_2})- (\frac{y_1}{\ell_1},\frac{y_2}{\ell_2})\right|^{\frac{n-1}{2}}} \\
&= &\frac{1}{|I_1||I_2|}\int_{\frac{I_1}{\ell_1}\times \frac{I_2}{\ell_2}} \frac{\ell_1^{n_1} \ell_2^{n_2}\ dz_1 dz_2}{\left|(z_1,z_2)- (\frac{y_1}{\ell_1},\frac{y_2}{\ell_2})\right|^{\frac{n-1}{2}}} \\
&\le &\left|\frac{I_1}{\ell_1}\times \frac{I_2}{\ell_2}\right|=1.
\end{eqnarray*}
\end{proof}

We can improve this result using fractional  type Poincaré inequalities. The idea is to ``interpolate'' between the oscillation
$$ 
\avgint_{I_1\times I_2} |f - f_{I_1\times I_2}|
$$
and the right hand side of \eqref{eq:suma} inspired by the following one parameter $(1,1)$ fractional  type Poincaré inequality which is easy to derive:
\begin{equation}\label{eq:FracPI}
\avgint_Q |f(t)-f_Q|dt \leq c_n\,\ell(Q)^{\delta}
 \avgint_Q \int_Q \frac{|f(t)-f(s)|}{|t-s|^{n+\delta }}\,dt\, ds.
\end{equation}

The result is the following, 

\begin{proposition}\label{pro:1-1FPIproductSpaces}
Let $R=I_1\times I_2$ be a rectangle and $\delta_1,\delta_2 \in (0,1)$. Then there exist dimensional constants $c_{n_1},c_{n_2}>0$ such that 
\begin{eqnarray}\label{eq:sumaFract}
\qquad \avgint_{R} |f - f_{R}|
& \leq & c_{n_1}
\ell(I_1)^{\delta_1} \avgint_{R}  \int_{I_1} \frac{|f(x_1, x_2) - f(y_1, x_2) |}{|x_1-y_1|^{n_1+\delta_1}}\ dx_1 dy_1 dx_2  \\
&& + \ c_{n_2} \ell(I_2)^{\delta_2}\avgint_{R}  \int_{I_2} \frac{|f(y_1, x_2) - f(y_1, y_2) |}{|x_2-y_2|^{n_2+\delta_2}}\ dy_1 dx_2 dy_2. \nonumber 
\end{eqnarray}
\end{proposition}

\begin{proof}

By the triangle inequality 
\begin{eqnarray*}
\avgint_{R} |f - f_{R}|  & \approx &
\avgint_{R}\avgint_{R} |f(x_1,x_2)- f(y_1,y_2)| \, dy_1 dy_2\ dx_1 dx_2\\
& \leq &\avgint_{R}\avgint_{R} |f(x_1,x_2)- f(y_1,x_2)| \, dy_1 dy_2\ dx_1 dx_2\\
&& + \
\avgint_{R}\avgint_{R} |f(y_1,x_2)- f(y_1,y_2)| \, dy_1 dy_2\ dx_1 dx_2\\
&=&\avgint_{R}\avgint_{I_1} |f(x_1,x_2)- f(y_1,x_2)| \, dy_1 \ dx_1 dx_2\\
&&
+ \
\avgint_{R}\avgint_{ I_2} |f(y_1,x_2)- f(y_1,y_2)| \, dy_1 dy_2\  dx_2=A+B.
\end{eqnarray*}

We compute $A$ to derive the first term of \eqref{eq:sumaFract}. A similar estimate holds for $B$ to derive the second term of \eqref{eq:sumaFract}
\begin{eqnarray*}
A & = &  \avgint_{R}\avgint_{ I_1}    \frac{|f(x_1,x_2)- f(y_1,x_2)| |x_1-y_1|^{\delta_1} }{|x_1-y_1|^{\delta_1}}
\, dy_1 dx_1\  dx_2\\
& \leq & \ell(I_1)^{\delta_1} \avgint_{R}\avgint_{ I_1}    \frac{|f(x_1,x_2)- f(y_1,x_2)| }{|x_1-y_1|^{\delta_1}}
\, dy_1 dx_1\  dx_2 \\
& \leq & c_{n_1} \, \ell(I_1)^{\delta_1} \avgint_{R} \int_{I_1}   \frac{|f(x_1,x_2)- f(y_1,x_2)|  }{|x_1-y_1|^{n_1+\delta_1}}
\, dy_1 dx_1\  dx_2. 
\end{eqnarray*}

\end{proof}

We remark that the fractional functional right hand side of \eqref{eq:sumaFract}  is smaller than the gradient functional from \eqref{eq:suma}, so  we could have avoided the use of  Lu-Wheeden pointwise estimate from Proposition \ref{pro:LW} and in addition we are getting a better estimate than we were looking for.
Indeed, to verify   last assertion we use the one-parameter fractional estimate from the following lemma.
\begin{lemma}\label{lem:oneparameterFractNabla} %
There exists a dimensional constant $c_n>0$ such that for any \, $\delta\in(0,1)$, and any cube $Q$ in $\mathbb{R}^{n}$
\begin{equation}\label{eq:roughfractionalPI}
\ell(Q)^{\delta}  \avgint_Q \int_Q\frac{|f(t)-f(s)|}{|t-s|^{n+\delta  }}\,dt\,ds  \leq \frac{c_n}{\delta(1-\delta) }\ell(Q)\avgint_Q |\nabla f| \,dx.
\end{equation}
\end{lemma}

Since we do not find the proof of this in the literature we provide an argument in Appendix \ref{sec:App:oneparameterFractNabla}. 

\begin{proof} [Second proof of Lemma \ref{lem:(1,1)PI-ccR}]

Let $\delta_1=\delta_2 =\frac12 \in(0,1).$ We compare the first term in the fractional functional \eqref{eq:sumaFract}  with the first term of the gradient functional \eqref{eq:suma}. To verify this  we use the one-parameter fractional estimate from \eqref{eq:roughfractionalPI}: 

\begin{eqnarray*}
\ell(I_1)^{\delta_1} \avgint_{I_1\times I_2}  \int_{I_1} \frac{|f(x_1, x_2) - f(y_1, x_2) |}{|x_1-y_1|^{n_1+\delta_1}}\ dx_1 dy_1 dx_2 & = &\\
\\
& &\hspace{-6cm} =\ \ell(I_1)^{\delta_1} \avgint_{I_2}  \avgint_{I_1} \int_{I_1} \frac{|f(x_1, x_2) - f(y_1, x_2) |}{|x_1-y_1|^{n_1+\delta_1}}\ dx_1 dy_1 dx_2 \\
\\
& &\hspace{-6cm} \le \ \frac{c_{n_1}}{\delta_1(1-\delta_1)} \ell(I_1) \avgint_{I_2}  \avgint_{I_1} | \nabla_1 f| \ dx_1  dx_2.
\\
& &\hspace{-6cm} = \ 4c_{n_1} \ell(I_1) \avgint_{I_2}  \avgint_{I_1} | \nabla_1 f| \ dx_1  dx_2. 
\end{eqnarray*}
 A similar estimate holds for the second direction $I_2$ with $\delta_2$.  Then applying  
 Theorem \ref{thm:BBMbiparametrico}  with $p_1=p_2=1$ we are done  or directly from ``rough" Fractional Poincar\'e inequality  
 \eqref{eq:FracPI}.

\end{proof}

To conclude with the proof of the lemma, it remains to prove  Theorem \ref{thm:BBMbiparametrico}.
We recall again that  there is very interesting improvement of \eqref{eq:FracPI} obtained by Bourgain, Brezis and Mironescu  \cite{BBM}
% \carlos{and Maz'ja-Shaponiskova \cite{MS}  which was greatly   improved in M. Milman's work \cite{M} by completely more general methods.}\footnote{\carlos{sacar todo esto}}  
The result is the following.

\begin{theorem} \label{thm:BBM}  Let $\delta \in (0,1)$. Then there exists a dimensional constant $c_n>0$ such that
\begin{equation*}
\avgint_Q |f(x)-f_Q|dt \leq c_n\,  
%\delta^{\frac1p}
(1-\delta)^{\frac1p}\,
 \ell(Q)^{\delta} \,
 \left(\avgint_Q \int_Q \frac{|f(x)-f(y)|^p}{|x-y|^{n+\delta p}}\,dy\, dx\right)^{1/p}. 
\end{equation*}
for every cube $Q$ in $\R^n$.
\end{theorem} 

\begin{remark}
We remark that the case $p>1$ does not follow from the case $p=1$ using Jensen's Also recall that there is sharper model version in Theorem \ref{thm:FracSobGain}.
\end{remark}

Our contribution here is the generalization stated in Theorem \ref{thm:BBMbiparametrico}.

\begin{proof}  [Proof of Theorem \ref{thm:BBMbiparametrico}]

The proof follows the same steps as in Proposition \ref{pro:1-1FPIproductSpaces}, but using the more precise estimate from Theorem \ref{thm:BBM}.

Let $R=I_1\times I_2 \in \ccR$. By the triangle inequality 
\begin{align*}
\avgint_{R} |f - f_{R}|  &\approx
\avgint_{R}\avgint_{R} |f(x_1,x_2)- f(y_1,y_2)| \, dy_1 dy_2\ dx_1 dx_2\\
&\leq \avgint_{R}\avgint_{R} |f(x_1,x_2)- f(y_1,x_2)| \, dy_1 dy_2\ dx_1 dx_2\\
&+\avgint_{R}\avgint_{R} |f(y_1,x_2)- f(y_1,y_2)| \, dy_1 dy_2\ dx_1 dx_2\\
&= \avgint_{R}\avgint_{I_1} |f(x_1,x_2)- f(y_1,x_2)| \, dy_1 \ dx_1 dx_2\\
&+\avgint_{R}\avgint_{ I_2} |f(y_1,x_2)- f(y_1,y_2)| \, dy_1 dy_2\  dx_2\\
&:=A+B.
\end{align*}

We compute $A.$ A similar estimate holds for $B.$ Now, from Theorem \ref{thm:BBM} we have, 
\begin{equation*}%\label{eq:(1,p)BBM}
\avgint_{ Q }  \avgint_{ Q }|u(x)-u(y)|\, dx\,dy  \leq c_n\,
%\delta_1^{\frac{1}{p_1}}
(1-\delta_1)^{\frac{1}{p_1}} \,   \ell(Q)^{\delta_1} \,
\left(  \avgint_{Q}  \int_{Q}  \frac{|u(x)- u(y)|^{p_1}}{|x-y|^{n+p_1\delta_1}}\,dy\,dx\right)^{\frac{1}{p_1}}   
\end{equation*}
for every cube $Q$ in $\mathbb{R}^n$.   Applying this to the cube $I_1 \subset \mathbb{R}^{n_1}$
\begin{align*}
A&= \avgint_{R}\avgint_{ I_1}  |f(x_1,x_2)- f(y_1,x_2)|\, dy_1 dx_1\  dx_2\\
&=\avgint_{I_2}\avgint_{I_1}\avgint_{ I_1}  |f(x_1,x_2)- f(y_1,x_2)|\, dy_1 dx_1\  dx_2\\
&\le \avgint_{I_2}c_{n_1} \,
%\delta_1^{\frac{1}{p_1}} 
(1-\delta_1)^{\frac{1}{p_1}}\ell(I_1)^{\delta_1} \left(\avgint_{I_1} \int_{ I_1}  \frac{|f(x_1,x_2)- f(y_1,x_2)|^{p_1}}{|x_1-y_1|^{n_1+p_1\delta_1}}\, dy_1 dx_1\right)^{\frac{1}{p_1}}  dx_2\\
&
\leq c_{n_1} \,
%\delta_1^{\frac{1}{p_1}}
(1-\delta_1)^{\frac{1}{p_1}}\ell(I_1)^{\delta_1} \left(\avgint_{R}\int_{ I_1}  \frac{|f(x_1,x_2)- f(y_1,x_2)|^{p_1}}{|x_1-y_1|^{n_1+p_1\delta_1}}\, dy_1 dx_1\  dx_2\right)^{\frac{1}{p_1}}\,
\end{align*}
by Jensen's inequality. 
\end{proof}

The same idea could be used to obtain similar results involving $m$-fold products of cubes. Let $R$ be a rectangle of the form $R=\prod_{i=1}^m I_i$ in $\mathbb{R}^n$ written as a product of cubes with $I_i\subset \R^{n_i}$, $n=\sum_{i=1}^m n_i$ and $m\le n$. We fix some notation: the points in $\mathbb{R}^n$ are of the form $\x=(x_1,\dots,x_m)$, where each $x_i$ is itself a string of $n_i$ real numbers forming a vector in $\mathbb{R}^{n_i}$. For any $0\le i\le m$, we denote 
\begin{equation*}
\x_i=\left \{
\begin{array}{cc}
(x_1,\dots,x_m) & i=0\\
(y_1,\dots, y_i,x_{i+1},\dots,x_m) & 1\le i < m\\
(y_1,\dots,y_m) & i=m.
\end{array}
\right.
\end{equation*}
We also adopt the notation $d\x_i:=dy_1\dots dy_i dx_{i}\dots dx_n.$

\begin{theorem}\label{thm:BBM-multi-iparametrico}
Let $R=\prod_{i=1}^m I_i$ be a rectangle in $\mathbb{R}^n$ written as a product of cubes with $I_i\subset \R^{n_i},$ and $n=\sum_{i=1}^m n_i$ and $m\le n$. Let $0<\delta_i<1\leq p_i<\infty$, $1\le i\le m$. Then there are some dimensional constants $c_{n_i}$  such that, 
\begin{equation*}\label{eq:BBM-multi-iparametrico}
\avgint_{R} |f - f_{R}|\le 
\sum_{i=1}^m c_{n_i}  
%\delta_i^{\frac{1}{p_i}}
(1-\delta_i)^{\frac{1}{p_i}} \left( \avgint_{R}  \int_{I_i} \frac{|f(\x_{i-1}) - f(\x_i) |^{p_i}}{|x_i-y_i|^{n_i+p_i\delta_i}}\ dy_1\dots dy_i dx_{i}\dots dx_m\right)^{\frac{1}{p_i}}.
\end{equation*}

\end{theorem}

Going back to the simplest product case, namely the case of product of two cubes, in the next section we will prove the self-improving result in Theorem \ref{thm:AutomejorastrongccR}.

\section{Proofs part II: the  Biparameter  Poincaré-Sobolev  inequality  }\label{sec:Bi-parameterPSI}

The goal of this section is to prove Theorem \ref{thm:AutomejorastrongccR}, namely, for any $R\in \ccR$
\begin{equation}\label{eq:AutomejorastrongccR}
\left\|f-f_R\right \|_{L^{p^*}( \frac{w\, dx}{w(R)})}
\leq
C[w]_{A_{q,\ccR}}^{\frac{1}{p}} \left [ \ell(I_1) 
\left\|\nabla_1 f\right \|_{L^{p}( \frac{w\, dx}{w(R)})}+
\ell(I_2) 
\left\|\nabla_2 f\right \|_{L^{p}( \frac{w\, dx}{w(R)})}
 \right ],
\end{equation}
where %$\mu$ is the weighted measure given by $\mu=\frac{w\, dx}{w(R)}$ and 
$$\displaystyle \frac1p-\frac1{p^*}= \frac1{n}\frac1{q+\log [w]_{A_{q,\ccR}}}.$$

Since this result involves a weighted estimate for a weight in the class $A_{p,\ccR}$, let us first introduce its obvious definition adapted to the geometry of the basis $\ccR$. For  a weight $w$ in $\mathbb{R}^{n_1}\times\mathbb{R}^{n_2}$, $n=n_1+n_2$, we will say that $w \in A_{p,\ccR}$ if
\begin{equation}\label{eq:Ap-ccR}
[w]_{A_{p, \ccR}}:=\sup_{R \in \ccR} 
\left( \frac{1}{|R|}\int_R w(x)\,dx\right) \left(\frac{1}{|R|}\int_R w(x)^{-\frac{1}{p-1}}\,dx \right)^{p-1}< \infty. 
\end{equation}

and in the case $p=1$, for a finite constant $c$
\begin{equation}\label{eq:A1-ccR}
\frac{1}{|R|}\int_R w(x)\,dx \leq c\, \inf_R w    \qquad R \in \ccR
\end{equation}
and the smallest of the constants $c$  is denoted by \,$[w]_{A_{1, \ccR}}$.

The strong  $A_{\infty,\ccR}$  class is defined in the same way as in the cubic or strong case and it enjoys the same geometric conditions,  
$$ A_{\infty,\ccR}=\bigcup_{p>1}A_{p, \ccR}. $$

We emphasize here that in the context of $\ccR$, the main difficulty is not in the self-improving method. 
The cube-product structure does not get affected by dyadic decompositions, so the standard procedures can be used to obtain Calder\'on--Zygmund coverings and many consequences of C--Z. In fact an inspection of the proof of Theorem \ref{thm:AutomejorastrongcR}, Theorem \ref{thm:PoSo-Aq-diam-delta} and Theorem \ref{thm:PoSo-Aq-diam-delta-weak} will show that the dyadic analysis will produce the analogous results for the basis $\ccR$. So we leave to the interested reader to check the details to prove the following claim.

\begin{claim}\label{cla:ccR-versions} 
There are analogous versions of  Theorem \ref{thm:AutomejorastrongcR}, Theorem \ref{thm:AutomejoraweakcR}, Theorem \ref{thm:PoSo-Aq-diam-delta} and Theorem \ref{thm:PoSo-Aq-diam-delta-weak} in the context of $\ccR$ with the same hypothesis and conclusions with the obvious modifications.
\end{claim}

Hence, the real problem here is in to find a useful starting point and also to check appropriate $D_p$ or $SD_p^s$-like conditions for the involved functionals. The first problem was solved in Section \ref{sec:Prelim-ccR}, with two proofs for Lemma \ref{lem:(1,1)PI-ccR}. The second will be studied below, in Lemma \ref{lem:Dp-likeConditions-ccR}.

Regarding the starting point,  we have a sort of unweighted $(1,1)$-Poincar\'e inequality proved in Lemma \ref{lem:(1,1)PI-ccR} that can be used as a starting point for our self-improving method. In the same way as in \eqref{eq:HigherOrderPoincare-Weighted}, for a weight $w\in A_{p,\ccR}$ we can apply H\"{o}lder's inequality to \eqref{eq:suma} to obtain, for a rectangle $R=I_1\times I_2$, that
\begin{equation*}%\label{eq:sum-HolderTrick}
\avgint_{R} |f - f_{R}|\le [w]_{ A_{p,\ccR}}^{\frac1p} \,\ell(I_1)\left\|\nabla_1f\right \|_{L^p\left ( \frac{w\, dx}{w(R)}\right )}
+ [w]_{ A_{p,\ccR}}^{\frac1p}\,\ell(I_2)\left\|\nabla_2f\right \|_{L^p\left (  \frac{w\, dx}{w(R)} \right )}.
\end{equation*}
We can define the functionals
\begin{equation}\label{eq:ai}
a_i(R):= [w]_{ A_{p,\ccR}}^{\frac1p}\,\ell(I_i)\left\|\nabla_if\right \|_{L^p\left (\frac{w\, dx}{w(R)} \right )},
\qquad 
i=1,2
\end{equation}
to obtain our bi-parameter starting point
\begin{equation}\label{eq:sum-StartingPoint}
\avgint_{R} |f - f_{R}|\le a_1(R)+a_2(R).
\end{equation}

The novelty here is that our functional $a$ is a \emph{sum} of two functionals with certain structure which is \emph{not} exactly as in \eqref{eq:Model-a(R)}. The presence of the sidelenght instead of the diameter makes the situation non standard, so we need to find what kind of geometric condition is satisfied. We summarize all those properties in the following lemma.

\begin{lemma}\label{lem:Dp-likeConditions-ccR}

Let $R\in \ccR$ and let $R=I_1\times I_2$, where  $I_1\subset \mathbb{R}^{n_1}$ and $I_2\subset \mathbb{R}^{n_2}$ are cubes. Also put $n=n_1+n_2$. For a weight $w\in A_q$ with $1\le q\le p<n$, consider the functional 
\begin{equation}\label{eq:general-sidelength-a(R)}
a(R)= \ell(I_j)^{\delta}\left(\frac{1}{w(R)} \int_R A(R,x)\,dx \right)^{1/p} \qquad j=1,2
\end{equation}
where $A(R,x)$ is nonnegative and increasing in $R$, that is, $R_1 \subset R_2$ implies that  $A(R_1,x)\leq A(R_2,x)$. 

Then  
\begin{enumerate}
\item $a\in SD_{p,\ccR}^{\frac{n}{\delta}}(w)$ with norm less or equal than 1.
\item Consider the Sobolev-type exponent $p^*_w$ defined by the usual condition
$\frac{1}{p} -\frac{1}{ p_w^*}=\frac{\delta}{nq}$. Then $a\in D_{p_w^{*},\ccR}(w)$ with $\|a\|_{D_{p_w^*,\ccR}}\leq 
[w]^\frac{\delta}{nq}_{A_{q,\ccR}}. $
\item  For  a given $M>1$,  consider now the Sobolev-type exponent $p^*_w$ defined by  
$\frac{1}{p} -\frac{1}{ p_w^*}=\frac{\delta}{nq}\frac1M$.  Then $a\in SD^s_{p_w^*,\ccR}(w)$ with  $s=\frac{ nM'}{\delta}$ and $\|a\|_{SD^s_{p_w^*,\ccR}(w)}\leq [w]_{A_{q,\cR}}^{\frac{\delta}{nqM} }$.
\end{enumerate}

\end{lemma}

\begin{proof}

The arguments are very similar to those given before, so we will omit some of details of the proof, but mention the key idea behind them. As in  all the previous cases, it is crucial to quantify the eccentricity of the objects that we are dealing with. Motivated by the fact that $|R|=\ell(I_1)^{n_1}\ell(I_2)^{n_2}$, we will define an \emph{eccentricity related} quantity as
\begin{equation*}
E(R)=\frac{\ell(I_2)^{n_2}}{\ell(I_1)^{n_2}}.
\end{equation*}
This is equivalent to the expression 
$$
\ell(I_1)^n.E(R)=|R|.
$$
In other words, the quantity $E(R)$ reflects the relation between the volume of the $n$-dimensional cube built from the side length $\ell(I_1)$ with respect to the actual measure of the rectangle $R$. 
The crucial property here is that this quantity is invariant with respect to dyadic children from $R$, namely 
as in Lemma \ref{lem:eccentr}, $E(\tilde{R})=E(R)$ for any $\tilde{R}$ dyadic descendant of $R \in \ccR$

With this idea in mind and using the quantity $E(R)$, the proof of (1) is a variation of the argument given in Example \ref{ex:modelexample}, the proof of (2) can be translated from the proof of \eqref{eq:p1*-importantEstimate} and finally the proof of (3) is an appropriate variation of  the proof of Lemma \ref{lem:p^*}.

Indeed, consider a family of disjoint dyadic subrectangles from $R\in \ccR$ denoted by $\{R_i\}$. Of course, each subrectangle is of the form $R_i=I_i\times J_i$ where $I_i$ and $J_i$ are dyadic subcubes of $I$ and $J$ respectively.
\end{proof}

\begin{remark}\label{rem:SumOfTwo} To be able to apply our main Theorem, the easiest way to proceed is to note that any $D_p$-like condition enumerated in Lemma \ref{lem:Dp-likeConditions-ccR}, is preserved when summing two functionals. This can be verified as follows (for the unweighted $D_{p,\ccR}$ as an example, the other situations are similar). Let $\{R_i\}_i$ be a family of pairwise disjoint dyadic subrectangles of $R=I_1\times I_2.$ Then, using Minkowsky inequality,
\begin{eqnarray*}
\left(\sum_i (a_1(R_i)+ a_2(R_i))^p \frac{|R_i|}{|R|}\right)^\frac{1}{p}&\le &\left(\sum_i a_1(R_i)^p \frac{|R_i|}{|R|}\right)^\frac{1}{p}  \\
&&+ \left(\sum_i a_2(R_i)^p \frac{|R_i|}{|R|}\right)^\frac{1}{p}\\
&\le & \|a_1\|_{D_{p,\ccR}} a_1(R) + \|a_2\|_{D_{p,\ccR}} a_2(R)\\
&\le &C\left (a_1(R) + a_2(R)\right ),
\end{eqnarray*}
with $C=\max\left \{\|a_1\|_{D_{p,\ccR}},\|a_2\|_{D_{p,\ccR}}\right \} $. 
\end{remark}

We can now proceed to present the proof of Theorem \ref{thm:AutomejorastrongccR}. 

\begin{proof}[Proof of Theorem \ref{thm:AutomejorastrongccR}]

Suppose first that $w$ is a weigh such that $[w]_{A_{q,\ccR}} \geq e^q$. Since the functionals from \eqref{eq:ai} namely
\begin{equation*}
a_i(R):= [w]^{\frac1p}_{ A_{p,\ccR}}\ell(I_i)\left\|\nabla_if\right \|_{L^p\left (\mu\right )},
\qquad 
i=1,2
\end{equation*}
satisfy a $ SD_{p_w^*,\ccR}^s$ condition given in Lemma \ref{lem:Dp-likeConditions-ccR} (with $\delta=1$), (3), the same holds for 
\begin{equation*}
a(R):= a_1(R)+a_2(R).
\end{equation*}
Hence, we can apply the $\ccR$-version of Theorem \ref{thm:PoSo-Aq-diam-delta} to this functional  to obtain the proof with a uniform bound using the assumption on $w$. More precisely, we obtain that 
\begin{equation*}
\left\|f-f_R\right \|_{L^{p^*}( \frac{w\, dx}{w(R)})}
\leq c_{p,q}
[w]_{A_{p,\ccR}}^{\frac{1}{p}} \left [ \ell(I_1) 
\left\|\nabla_1 f\right \|_{L^{p}( \frac{w\, dx}{w(R)})}+
\ell(I_2) 
\left\|\nabla_2 f\right \|_{L^{p}( \frac{w\, dx}{w(R)})}
 \right ].
\end{equation*}

%where $\mu$ is the weighted measure given by $\mu=\frac{w\, dx}{w(R)}$.

Now suppose that we are dealing with a trivial weight (such as $[w]_{A_{q,\ccR}}=1$) or more generally with flat weights such that $[w]_{A_{q,\ccR}}\le e^q$. To avoid the blowup, we can apply the $\ccR$-version of Theorem \ref{thm:AutomejoraweakcR} to get the weak inequality
\begin{equation*}
\| f-f_R\|_{L^{p^*_w,\infty}\big( R, \frac{w\,dx}{w(R)}\big)} \le
c\, p^*_{1}\,e^q\, e^{ \frac{1}{n}  } \, a(R),
 \end{equation*}
where $ \frac{1}{p} -\frac{1}{ p_1^*}=\frac{1}{nq}$.
To finish with the proof, we can use a bi-parameter version of the truncation method to jump to the strong bound (see Appendix \ref{sec:App-BiParameterTruncation}).
\end{proof}

\section{Proofs part III: the local Biparameter  Fractional Poincaré-Sobolev  inequality  }
\label{sec:Bi-parameterFractionalBBMPSI}

 In this section we will provide a proof of Theorem \ref{thm:AutomejorastrongBBM-A1-ccR}. We use the main argument from Section \ref{sec:Bi-parameterPSI} for proving Theorem \ref{thm:AutomejorastrongccR} with a different functional $a(R)$ given by Theorem  \ref{thm:BBMbiparametrico}.

\begin{proof}[Proof of Theorem \ref{thm:AutomejorastrongBBM-A1-ccR}]

From Theorem \ref{thm:BBMbiparametrico}  with $p_1=p_2=p$ we have 
\begin{equation*}
\avgint_{R} |f - f_{R}|\leq A+B
\end{equation*}
where 
\begin{equation*}
A\leq c_{n_1} 
%\delta^{\frac1p}
(1-\delta)^{\frac1p}\ell(I_1)^{\delta} \left(\avgint_{R}\int_{ I_1}  \frac{|f(x_1,x_2)- f(y_1,x_2)|^p}{|x_1-y_1|^{n_1+p\delta}}\, dy_1 dx_1\  dx_2\right)^{\frac1p}
\end{equation*}
and 
\begin{equation*}
B\leq c_{n_2} \, 
%\delta^{\frac1p}
(1-\delta)^{\frac1p} \ell(I_2)^{\delta}
\left(\avgint_{R}  \int_{I_2} \frac{|f(y_1, x_2) - f(y_1, y_2) |^p}{|x_2-y_2|^{n_2+p\delta}}\ dy_1 dx_2 dy_2.\right)^{\frac1p}
\end{equation*}

Now, recalling the definition of $A_{1, \ccR}$ given in \eqref{eq:A1-ccR}: 
\begin{equation*}%\label{eq:A1-ccR}
\frac{1}{|R|}\int_R w(x)\,dx \leq [w]_{A_{1, \ccR}}\, \inf_R w    \qquad R \in \ccR
\end{equation*}
we have, 
$$
A\leq c_{1}\,[w]_{A_{1, \ccR}}^{\frac1p}\, (1-\delta)^{\frac1p}\ell(I_1)^{\delta}
\left( \frac{1}{w(R)} \int_{R} \int_{ I_1}  \frac{|f(x_1,x_2)- f(y_1,x_2)|^p}{|x_1-y_1|^{n_1+p\delta}}\, w(x_1,x_2) dx_1  dx_2\ dy_1\right)^{\frac1p}.
$$

Similarly we have for $B$
$$
B\leq c_{2}\,[w]_{A_{1, \ccR}}^{\frac1p}\, (1-\delta)^{\frac1p}\ell(I_2)^{\delta} 
\left( \frac{1}{w(R)} \int_{R} \int_{ I_2}  \frac{|f(y_1,x_2)- f(y_1,y_2)|^p}{|x_2-y_2|^{n_2+p\delta}}\, w(y_1,y_2)\,  dy_1dy_2\ dx_2  \right)^{\frac1p}
$$
Let 
$$
a_1(R)= c_{1}\,[w]_{A_{1, \ccR}}^{\frac1p}\, (1-\delta)^{\frac1p}\ell(I_1)^{\delta} 
\left( \frac{1}{w(R)} \int_{R} \int_{ I_1}  \frac{|f(x_1,x_2)- f(y_1,x_2)|^p}{|x_1-y_1|^{n_1+p\delta}}\, w(x_1,x_2) dx_1  dx_2\ dy_1\right)^{\frac1p}.
$$
and 
$$
a_2(R) =c_{2}\,[w]_{A_{1, \ccR}}^{\frac1p}\, (1-\delta)^{\frac1p}\ell(I_2)^{\delta} 
\left( \frac{1}{w(R)} \int_{R} \int_{ I_2}  \frac{|f(y_1,x_2)- f(y_1,y_2)|^p}{|x_2-y_2|^{n_2+p\delta}}\, w(y_1,y_2)\,  dy_1dy_2\ dx_2  \right)^{\frac1p}
$$
Observe that each  functional $a_j$ is of the form \eqref{eq:general-sidelength-a(R)} since each of the inner integrand  of the functional, $A(R,x)$ is increasing in $R$,  
$$R_1=I_1\times I_2 \subseteq R_2= J_1\times J_2 \quad \iff  \quad I_1\subseteq J_1 \text{ and } I_2 \subseteq J_2$$
and we can apply  Lemma \ref{lem:Dp-likeConditions-ccR}.  Then if $M>1$,   $a_j \in SD^s_{p_w^*,\ccR}(w)$ with  $s=\frac{ nM'}{\delta}$ and $\|a\|_{SD^s_{p_w^*,\ccR}(w)}\leq [w]_{A_{q,\cR}}^{\frac{\delta}{nqM} }$.

Since the functionals from \eqref{eq:general-sidelength-a(R)} satisfy the $ SD_{p_w^*,\ccR}^s$ condition given in Lemma \ref{lem:Dp-likeConditions-ccR}, (3), namely 
\begin{equation*}
\frac{1}{p} -\frac{1}{ p_w^*}=\frac{\delta}{n}\frac1M.
\end{equation*}

\begin{equation*}%\label{eq:Keyestimate}
\left( \sum_{i}a_j(R_i)^{ p^*_{w} }\,\frac{w(R_i)}{w(R)}\right)^{\frac{1}{p^*_{w}}}  \leq  [w]_{A_{1,\ccR}}^{ \frac{ \delta}{nM} }
\,\left (  \frac{|\bigcup_iR_i|}{|R|}   \right )^{ \frac{\delta}{nM'}  } a(R)  \qquad j=1,2
\end{equation*}
and hence the corresponding result holds for 
\begin{equation*}
a(R):= a_1(R)+a_2(R),
\end{equation*}
more precisily 
\begin{equation*}%\label{eq:Keyestimate}
\left( \sum_{i} (a(R_i))^{ p^*_{w} }\,\frac{w(R_i)}{w(R)}\right)^{\frac{1}{p^*_{w}}}  \leq 2 [w]_{A_{1,\ccR}}^{ 
\frac{ \delta } {nM} 
}
\,\left (  \frac{|\bigcup_iR_i|}{|R|}      \right )^{ \frac{\delta}{nM'}  } a(R)  \qquad j=1,2
\end{equation*}

Now we are in a position to finish the proof  exactly as of Theorem \ref{thm:AutomejorastrongccR}. We omit the details.

\end{proof}

\appendix

\section{The truncation method} \label{sec:App-Truncation}

We include here for completeness the different truncation arguments that we used in our results.

\subsection{The classical weak implies strong for the gradient}\label{sec:App-ClassicalTruncation}

We include here the classical truncation argument for Lipschitz functions related to the functional defined by the integral of the gradient. 
The general idea of truncation arguments in the context of Poincar\'e-Sobolev inequalities
is well known, classical references are \cite{Mazya-Sobolev} and \cite{BCLSC}.

\begin{lemma}\label{lem:App-Classic-Truncation}
Let $g:   \mathbb{R}^{n}\to\R$ be any nonnegative Lipschitz function. Suppose that for the pair $1\le q<p$ there is a weak 
type estimate for a pair of measures $\nu, \mu$ of the form:
\[
\|g\|_{L^{p,\infty}_{\mu}}
\lesssim\left (\int_{\R^{n}} |\nabla g|^q\, d\nu\right )^\frac{1}{q}.
\]
Then the strong estimate also holds, namely
\[
\|g\|_{L^p_{\mu}} \lesssim\left (\int_{\R^{n}} |\nabla g|^q\, d\nu\right )^\frac{1}{q}.
\]
\end{lemma}

\begin{proof}
Consider the usual truncation of a non negative function $g$ at level $2^k$ given by $T_k(g)$ defined by:
\begin{equation*}
T_kg(x):=
\left \{\begin{array}{cc}
0& \text{ if } g(x)\le 2^k\\
g(x) - 2^k& \text{ if } 2^k< g(x) < 2^{k+1}\\
2^k & \text{ if } g(x)\ge 2^{k+1}.\\
\end{array}
\right .
\end{equation*}
Also define for each $k\in \mathbb Z$ the set $G_k:=\{x\in \mathbb{R}^n:  2^k< g(x)\le 2^{k+1}\}$.
We have that, for all $x\in G_{k+1}$, $T_kg(x)=2^k$ and $\text{sop}\nabla(T_kg)\subset G_k$. We proceed as follows:
\begin{eqnarray*}
\int_{\mathbb{R}^n}g^p(x) d\mu  &\lesssim &\sum_{k=-\infty}^{k=\infty} 2^{kp}\mu(G_{k+1})\\
& \lesssim  & \sum_{k=-\infty}^{k=\infty} 2^{(k-1)p}\mu( T_kg(x)>2^{k-1})\\
& \lesssim  & \sum_{k=-\infty}^{k=\infty} \left (2^{k-1}\mu( T_kg(x)>2^{k-1})^\frac{1}{p}\right )^{p}.
\end{eqnarray*}
Since $T_kg$ is still a Lipschitz function, we can use the hypothesis and get
\begin{eqnarray*}
\int_{\mathbb{R}^n}g^p(x) d\mu  &\lesssim &\sum_{k=-\infty}^{k=\infty} \left (\int_{\R^{n}} |\nabla T_kg|^q\, d\nu\right )^\frac{p}{q}\\
&\lesssim &\sum_{k=-\infty}^{k=\infty} \left (\int_{G_k} |\nabla T_kg|^q\, d\nu\right )^\frac{p}{q}\\
&\lesssim &\sum_{k=-\infty}^{k=\infty} \left (\int_{G_k} |\nabla g|^q\, d\nu\right )^\frac{p}{q}\\
&\lesssim & \left (\int_{\mathbb{R}^n} |\nabla g|^q\, d\nu\right )^\frac{p}{q}
\end{eqnarray*} 
using that $\frac{p}{q}>1$ and the disjointness of the family $\{G_k\}$.
\end{proof}

\subsection{Truncation in the bi-parameter setting}
\label{sec:App-BiParameterTruncation}

\begin{lemma}\label{lem:App-Bi-parameter-Truncation}
Let $g:   \mathbb{R}^{n_1}\times \mathbb{R}^{n_2} \to\R$ be any nonnegative Lipschitz function. Suppose that for $p\ge1$ there is a weak 
$(q,p)$-type estimate for a pair of measures $\nu, \mu$ of the form:
\[
\|g\|_{L^{p,\infty}_{\mu}}
\lesssim\left (\int_{\R^{n_1}\times\R^{n_2}} (|\nabla_1g|^q d\nu\right )^{1/q}+\left (\int_{\R^{n_1}\times\R^{n_2}} (|\nabla_2g|^q d\nu\right )^{1/q}.
\]
Then the strong estimate also holds, namely
\[
\|g\|_{L^p_{\mu}}
\lesssim\left (\int_{\R^{n_1}\times\R^{n_2}} (|\nabla_1g|^q d\nu\right )^{1/q}+\left (\int_{\R^{n_1}\times\R^{n_2}} (|\nabla_2g|^q d\nu\right )^{1/q}.
\]
\end{lemma}
\begin{proof}
This case is as easy as Lemma \ref{lem:App-Classic-Truncation}, since the class of Lipschitz functions enjoys the truncation property. The same proof works here as well.
\end{proof}

\section{The proof of Lemma \ref{lem:oneparameterFractNabla}}\label{sec:App:oneparameterFractNabla}

We include here the proof of Lemma \ref{lem:oneparameterFractNabla}, namely estimate \eqref{eq:roughfractionalPI}. We recall here that estimate:

\begin{equation}\label{eq:fractional_derivative}
\ell(Q)^\delta \avgint_Q\int_Q\frac{|f(x)-f(y)|}{|x-y|^{n+\delta }}dydx 
\lesssim_{n} 
\frac{1}{\delta(1-\delta)}
\ell(Q)  \avgint_Q |\nabla f|.
\end{equation}

\begin{proof}[Proof of Lemma \ref{lem:oneparameterFractNabla}]

We start by using the FTC to obtain a representation formula as follows

Using the FTC, for every $x,y\in Q$, 
\[
f(y)-f(x)=\int_0^1\nabla f(x+t(y-x))\cdot(y-x)dt.
\]

Then, for a fixed $x\in Q$, the inner integral in \eqref{eq:fractional_derivative} can be written as

\begin{eqnarray*}
\int_Q\frac{|f(x)-f(y)|}{|x-y|^{n+\delta }}dydx
& \leq  & \int_Q\int_0^1\frac{|\nabla f(x+t(y-x))|}{|x-y|^{n+\delta -1}}dtdydx\\
& = & \int_0^1\int_{Q\cap B(x,\sqrt{n}\ell(Q))} \frac{|\nabla f(x+t(y-x))|}{|x-y|^{n-(1-\delta)}}dydtdx\\
& = & I
\end{eqnarray*}
since $Q\subset B(x,\sqrt{n}\ell(Q))$ for any $x\in Q$. Now, we change variables putting $z=x+t(y-x)=(1-t)x+ty$. Then one has $|x-y|=|z-x|/t$ and the change of variables theorem  yields 
\begin{eqnarray*}
I & = & \int_0^1\int_{ ((1-t)x+tQ) \cap B(x,\sqrt{n}t\ell(Q))} \frac{|\nabla f(z)|}{|z-x|^{n-(1-\delta)}}\frac{t^{n-(1-\delta)}}{t^n}dzdtdx\\
& \leq & \int_0^1\int_{Q\cap B(x,\sqrt{n}t\ell(Q))} \frac{|\nabla f(z)|}{|z-x|^{n-(1-\delta)}}t^{-(1-\delta)}dzdtdx\\
&\le & \int_Q\int_{\frac{|z-x|}{\sqrt{n}\ell(Q)}}^1 \frac{dt}{t^{1-\delta}} \frac{|\nabla f(z)|}{|z-x|^{n-(1-\delta)}}dzdx
\end{eqnarray*}
by Fubini again. Going back to the beginning and putting all together we obtain
\begin{equation*}
 \frac{\ell(Q)^\delta }{\delta} \,  \avgint_Q\int_Q   \frac{|\nabla f(z)|}{|z-x|^{n-(1-\delta)}}dzdx
= \frac{\ell(Q)^\delta }{\delta} \,\avgint_Q|\nabla f(z)| \int_Q  \frac{dx}{|z-x|^{n-(1-\delta)}} dz.
\end{equation*}
Now we use that, for any Lebesgue measurable set $\Omega$ and $0<\alpha<n$, we have the estimate
\begin{equation*}\label{eq:lemita}
\int_\Omega  \frac{dx}{|z-x|^{n-\alpha}} \leq v_n^{-\alpha/n}\alpha^{-1}|\Omega|^{\alpha/n},\qquad \text{for all }z,
\end{equation*}
where $v_n$ is the volume of the unit ball of $\mathbb{R}^n$. We finally obtain that 
\begin{eqnarray*}
\ell(Q)^\delta \avgint_Q\int_Q\frac{|f(x)-f(y)|}{|x-y|^{n+\delta }}dydx & \le &
c_n \ell(Q)^\delta \frac{1}{\delta(1-\delta)}\, |Q|^\frac{1-\delta}{n} \, \avgint_Q|\nabla f(z)| \\
&= & \frac{c_n}{\delta(1-\delta)}\, \ell(Q) \, \avgint_Q|\nabla f(z)|.
\end{eqnarray*}
\end{proof}

\bibliographystyle{amsalpha}

%

%\bibliography{refes}
%\bibliography{/Users/ezequielrela/Dropbox/Eze/refes}%Mac_casa
%\bibliography{/Users/ezequiel/Dropbox/Eze/refes}%Oficina2077

\end{document}